\documentclass[11pt]{amsart}

\usepackage[T1]{fontenc}
\usepackage[utf8]{inputenc}
\usepackage{lmodern}
\usepackage[margin=1.08in]{geometry}
\usepackage{amsmath,amssymb,amsthm,mathtools}
\usepackage{microtype}
\usepackage{enumitem}
\usepackage{xcolor}
\usepackage{tikz}
\usepackage[colorlinks=true,linkcolor=blue,citecolor=green!45!black,urlcolor=blue]{hyperref}

\newcommand{\GL}{\operatorname{GL}}
\newcommand{\Mat}{\operatorname{Mat}}
\newcommand{\tr}{\operatorname{tr}}
\newcommand{\Ind}{\operatorname{Ind}}
\newcommand{\Hom}{\operatorname{Hom}}
\newcommand{\cO}{\mathcal O}
\newcommand{\one}{\mathbf 1}
\newcommand{\pord}{\operatorname{pord}}
\newcommand{\len}{\ell}
\newcommand{\MW}{\mathrm{MW}}

\theoremstyle{plain}
\newtheorem{theorem}{Theorem}[section]
\newtheorem{proposition}[theorem]{Proposition}
\newtheorem{lemma}[theorem]{Lemma}
\newtheorem{corollary}[theorem]{Corollary}
\newtheorem{conjecture}[theorem]{Conjecture}

\theoremstyle{definition}
\newtheorem{definition}[theorem]{Definition}

\theoremstyle{remark}
\newtheorem{remark}[theorem]{Remark}

\title{A degenerate Whittaker criterion for $\mathrm GL_{2n}$}
\author{Taiwang  Deng}
\address[T.D.]{Beijing Institute of Mathematical Sciences and Applications (BIMSA), Huairou District, 100084, Beijing\\
China}
\email{dengtaiw@bimsa.cn}
\date{}

\begin{document}

\begin{abstract}
Let $F$ be a non-Archimedean local field.  Let $N$ be the unipotent radical of
the standard parabolic subgroup of $\mathrm GL_{2n}(F)$ of type $(n,n)$ with fixed
nondegenerate additive character $\psi$.  For an irreducible admissible
representation $\pi$ of $\mathrm GL_{2n}(F)$, a theorem due to
Gomez--Gourevitch--Sahi on generalized Whittaker models gives a criterion for
the vanishing of the twisted Jacquet module $\pi_{N,\psi}$ in terms of the
wave-front set.  We translate this orbit-theoretic answer into
Langlands--Zelevinsky data: if $\pi=L(\mathfrak m)$, then
$\pi_{N,\psi}=0$ if and only if the Zelevinsky dual
$\mathfrak m^{\mathrm t}$ contains a segment of length at least $n+1$.

We do this in response to a conjecture proposed by D. Prasad about the
vanishing of $\pi_{N,\psi}$ in terms of the adjoint $L$-function
$L(s,\pi\times\pi^\vee)$.  We prove that, for every irreducible
representation $\pi$, vanishing of $\pi_{N,\psi}$ implies the pole inequalities
predicted by D.Prasad.  However, we show that the converse implication is
false by an explicit counterexample for $\mathrm GL_4(F)$.

For the generalized Steinberg constituents $v_{P_\beta}^G$ of the principal
series containing the trivial representation, we make an explicit calculation
of when $\pi_{N,\psi}$ is zero.  In particular, for $\GL_6(F)$, exactly three
of the $32$ constituents of such a principal series violate the converse
direction of the conjecture proposed by D. Prasad.
\end{abstract}

\maketitle

\section*{Introduction}

Let $F$ be a non-Archimedean local field and let $\pi$ be an irreducible
admissible representation of \(G=\GL_{2n}(F)\).  Let \(P_{n,n}=MN\) be the
standard parabolic subgroup with Levi factor
\(M\simeq\GL_n(F)\times\GL_n(F)\), and write
\[
  n(X)=\begin{pmatrix}I_n&X\\0&I_n\end{pmatrix},
  \qquad
  \psi(n(X))=\psi_0(\operatorname{Tr}X).
\]
The object studied in this paper is the twisted Jacquet module
\(\pi_{N,\psi}\).  This is the degenerate Whittaker model attached to the
nilpotent orbit \((2^n)\), and it is naturally a representation of the diagonal
copy \(\Delta\GL_n(F)\), the stabilizer of \(\psi\) in \(M\).

Our primary goal is to address a conjecture of D.~Prasad, stated in
\cite[Conjecture 6.3]{HV26}.  Write
\[
 \pord_{s=t}L(s,\pi\times\pi^\vee):=\text{ the order of pole at \(s=t\) of the adjoint \(L\)-function } L(s,\pi\times\pi^\vee).
\]
The conjecture predicts that
\begin{equation}
\label{eq:intro-prasad}
  \pi_{N,\psi}=0
  \quad\Longleftrightarrow\quad
  \pord_{s=t}L(s,\pi\times\pi^\vee)\geq n-t+1
  \qquad(1\leq t\leq n).
\end{equation}
In other words, the conjecture asks whether the
vanishing of \(\pi_{N,\psi}\) can be detected directly from a
finite list of pole orders of \(L(s,\pi\times\pi^\vee)\).  The motivation is
parallel to the Gross--Prasad--Rallis philosophy, where poles of adjoint
\(L\)-functions are expected to detect genericity phenomena in \(L\)-packets.

The first answer is in terms of the wave-front set.  The closure theorem for generalized
Whittaker models of Gomez--Gourevitch--Sahi \cite[Theorem E]{GGS17} implies
that \(\pi_{N,\psi}\) is nonzero exactly when the orbit \((2^n)\) lies in the
closure of the wave-front set of \(\pi\).  If \(\lambda(\pi)\) denotes the
partition attached to the wave-front set of \(\pi\),
Theorem~\ref{thm:orbit-criterion} gives the criterion
\[
  \pi_{N,\psi}=0
  \quad\Longleftrightarrow\quad
  \bigl(\lambda(\pi)^\top\bigr)_1\geq n+1.
\]
Thus the twisted Jacquet module vanishes precisely when the transpose of the
partition attached to the wave-front set has a part larger than \(n\).

The next step is to express this criterion in terms of the Langlands
classification.  Write \(\pi=L(\mathfrak m)\) as a Langlands quotient, where
\(\mathfrak m\) is a multisegment.  The partition attached to the wave-front set is
obtained through the Zelevinsky involution, or equivalently the
M\oe glin--Waldspurger algorithm.
If \(r_{\MW}(\mathfrak m)\) denotes the largest length of a segment in the
Zelevinsky dual multisegment \(\mathfrak m^{\mathrm t}\), then
Theorem~\ref{thm:corrected} gives
\begin{equation}
\label{eq:zelevinsky-dual-criterion}
  \pi_{N,\psi}=0
  \quad\Longleftrightarrow\quad
  r_{\MW}(\mathfrak m)\geq n+1.
\end{equation}
This is the main structural criterion of the paper.  It also explains why the
comparison with an \(L\)-function criterion is subtle: the vanishing problem is
governed by the Zelevinsky dual, whereas the Rankin--Selberg factors are
naturally computed from the original multisegment.
As two useful complements, Theorem~\ref{thm:vanishing-parabolic-subquotient}
realizes every representation with \(\pi_{N,\psi}=0\) as a subquotient of
\(\pi_1\times\pi_2\) with \(\pi_1\) one-dimensional, while
Corollary~\ref{cor:unique-nonzero-parabolic-subquotient} computes the
\((N,\psi)\)-coinvariants of such an induction and singles out its unique
irreducible subquotient with nonzero twisted Jacquet module.
Section~\ref{sec:further-examples} records Speh, Arthur-type, and parabolic
induction examples whose vanishing is immediate from
\eqref{eq:zelevinsky-dual-criterion}.  Note that these examples are not
accessible by the methods of \cite{HV26}.

The comparison with Prasad's proposed criterion has the following outcome.
The implication from vanishing to the pole inequalities holds.  More
generally, Theorem~\ref{thm:general-vanishing-to-poles} proves that for
\(\pi=L(\mathfrak m)\) and every positive integer \(t\),
\[
  \pord_{s=t}L(s,\pi\times\pi^\vee)
  \geq \max\{r_{\MW}(\mathfrak m)-t,0\}.
\]
Thus for any admissible irreducible representation \(\pi\) of
\(\GL_{2n}(F)\), the condition \(\pi_{N,\psi}=0\) implies all pole
inequalities in \eqref{eq:intro-prasad}.  The converse implication is false
with an explicit counterexample in \(\GL_4(F)\).

For the constituents attached to the normalized principal series
\(\Sigma_m(\rho)\) of \(\GL_m(F)\), defined by \eqref{eq:sigma-m-rho}, the
vanishing criterion \eqref{eq:zelevinsky-dual-criterion} can be made completely
explicit.  If \(P_\beta\) has Levi size given by 
ordered partition \(\beta\) and \(v_{P_\beta}^G\) denotes the generalized
Steinberg quotient of \(C^\infty(G/P_\beta)\), 
Corollary~\ref{cor:wf-principal-series} gives the wave-front set
\[
  \operatorname{WF}(v_{P_\beta}^G)
  =\overline{\mathcal O_{(\beta^\downarrow)^\top}}^{\,\mathrm{an}},
\]
the analytic closure of the Richardson orbit of \(P_\beta\).

The same principal-series computation supplies an infinite family of
counterexamples to the pole-to-vanishing implication in Prasad's conjecture.
For \(n\geq3\), let
\[
  \alpha_n=(\underbrace{1,\dots,1}_{n-1},2,
  \underbrace{1,\dots,1}_{n-1})
\]
be an ordered partition of \(2n\).  Corollary
\ref{cor:principal-series-pole-counterexamples} shows that the corresponding
constituent \(\pi_{\alpha_n}\) satisfies all pole inequalities in
\eqref{eq:intro-prasad}, but \((\pi_{\alpha_n})_{N,\psi}\neq0\).  In rank
\(\GL_6(F)\), Proposition~\ref{prop:GL6-three-failures} gives a complete
enumeration: among the \(32\) constituents of \(\Sigma_6(\rho)\), exactly
three violate the proposed equivalence.

Section~\ref{sec:exterior-square} separates two different roles of
\(L\)-functions.  The exterior-square factor detects possibly twisted Shalika
functionals, but it is
too coarse for the twisted Jacquet module considered here.  By contrast, the
segment-length ratios defined in \eqref{eq:segment-length-ratio} are built
from Rankin--Selberg factors of the Zelevinsky dual and give an exact
\(L\)-function test.  Theorem~\ref{thm:segment-length-ratio-criterion} states
that \(\pi_{N,\psi}=0\) if and only if, on some cuspidal line, the
finite-difference quotient
\[
  \frac{L(s,\pi^{\mathrm t}\times\tau_{\rho,n+1}^{\vee})}
       {L(s,\pi^{\mathrm t}\times\tau_{\rho,n}^{\vee})}
\]
has a pole. 

Thus the theorem of Gomez--Gourevitch--Sahi supplies the decisive vanishing
criterion for \(\pi_{N,\psi}\) in terms of the wave-front set, while the main
work of the paper is to relate this criterion to Langlands parameters, and
hence to Rankin--Selberg pole counts.  This allows us to prove one direction
of Prasad's proposed criterion, produces explicit counterexamples to the
other, and gives replacement criteria in terms of \(L\)-functions that are
exact.

\subsection*{Acknowledgements}
\addtocontents{toc}{\protect\setcounter{tocdepth}{1}}
This paper grew out of a lecture given by Professor D.~Prasad at a conference co-organized by the author at BIMSA.  The author is deeply grateful to Professor
Prasad for carefully reading an early draft and for making many helpful
suggestions.  The author is supported by the National Natural Science Foundation
of China, No.~12401013.

\section{The wave-front set criterion}

Let
\[
  G=\GL_{2n}(F),\qquad
  P=P_{n,n}=MN,
\]
where $M\simeq \GL_n(F)\times\GL_n(F)$ and
\[
  N=\left\{
  n(X):=\begin{pmatrix}I_n&X\\0&I_n\end{pmatrix}
  :X\in\Mat_n(F)\right\}.
\]
Fix a nontrivial additive character $\psi_0:F\to\mathbb C^\times$ and define
\[
  \psi(n(X))=\psi_0(\operatorname{Tr}X).
\]
For a smooth representation $\pi$ of $G$, define the twisted Jacquet module
\[
  \pi_{N,\psi}
  :=\pi\Big/\bigl\langle n v-\psi(n)v:n\in N,\ v\in\pi\bigr\rangle.
\]

Nilpotent orbits in $\mathfrak{gl}_m$ are indexed by partitions of $m$. Given a partition $\lambda$, we denote by
$\cO_{\lambda}$ the corresponding nilpotent orbit.  We use
$\mu\leq\lambda$ for the dominance order or geometrically,
$\cO_\mu\subseteq\overline{\cO_\lambda}$.
For a partition \(\lambda=(\lambda_1,\lambda_2,\dots)\), its transpose
\(\lambda^\top\) is the partition whose \(j\)-th part is
\[
  (\lambda^\top)_j:=\#\{i:\lambda_i\geq j\}.
\]
Note that \(\lambda^\top\) can be obtained by transposing the Young diagram of
\(\lambda\).

We consider the wave-front set $\mathrm{WF}(\pi)$ of a representation $\pi$, which consists of the
analytic closures of the nilpotent orbits for which the corresponding
generalized Whittaker models are nonzero; see M\oe glin--Waldspurger
\cite[Introduction]{MW87} and Gomez--Gourevitch--Sahi \cite[1.1]{GGS17}.

\begin{lemma}[Uniqueness of the wave-front orbit for \(\GL_m\)]
\label{lem:unique-wf-orbit}
Let \(\pi\) be an irreducible admissible representation of \(\GL_m(F)\).
Then the wave-front set of \(\pi\) is the analytic closure of a single
nilpotent orbit.  Equivalently, among the nilpotent orbits whose generalized
Whittaker models are nonzero, there is a unique maximal one.
\end{lemma}

\begin{proof}
M\oe glin--Waldspurger attach to \(\pi\) two sets of nilpotent orbits:
the set \(\mathcal N_{\mathrm{tr}}(\pi)\) arising from the local character
expansion and the set \(\mathcal N_{\mathrm{Wh}}(\pi)\) arising from degenerate
Whittaker models.  By \cite[Th\'eor\`eme~I.16]{MW87}, their maximal elements
coincide.  For \(G=\GL_m\), \cite[Chapitre~II, \S II.2, Proposition,
p.~444]{MW87} proves that the maximal elements of
\(\mathcal N_{\mathrm{tr}}(\pi)\) are contained in a unique nilpotent orbit.
Thus \(\mathcal N_{\mathrm{Wh}}(\pi)\) has a unique maximal orbit, and the
wave-front set is the analytic closure of this orbit.
\end{proof}

For an irreducible admissible representation $\pi$ of $\GL_m(F)$, let
$\lambda(\pi)$ be the partition attached to its wave-front set, namely the
partition indexing the unique nilpotent orbit supplied by
Lemma~\ref{lem:unique-wf-orbit}.
\begin{lemma}
\label{lem:closure}
 For every partition $\mu$ of
$m$, the generalized Whittaker model of type $\mu$ is nonzero if and only if
\[
  \mu\leq\lambda(\pi).
\]
Moreover, the generalized Whittaker model of type $\lambda(\pi)$ is
one-dimensional.
\end{lemma}

\begin{proof}
The nonvanishing statement is the $\GL_m$ closure-order theorem for generalized
Whittaker models; see \cite[Theorem E]{GGS17}.  In the non-Archimedean case,
M\oe glin--Waldspurger identify the dimension of the degenerate Whittaker
model attached to a maximal orbit with the corresponding coefficient in the
local character expansion; see \cite[Corollaire~I.17]{MW87}.  For
\(G=\GL_m\), the proof of \cite[Chapitre~II, \S II.2, Proposition,
p.~445]{MW87} shows that this coefficient is \(1\).  See also the discussion
following \cite[Corollary G]{GGS17}.
\end{proof}

\begin{lemma}
\label{lem:character}
The twisted Jacquet module $\pi_{N,\psi}$ is the generalized Whittaker model
attached to the nilpotent orbit of Jordan type
\[
  (2^n)=(\underbrace{2,\dots,2}_{n\text{ times}}).
\]
\end{lemma}

\begin{proof}
Use the trace pairing to identify $\mathfrak{gl}_{2n}$ with its dual, and set
\[
  f=\begin{pmatrix}0&0\\ I_n&0\end{pmatrix},
  \qquad
  h=\begin{pmatrix}I_n&0\\0&-I_n\end{pmatrix}.
\]
Then $[h,f]=-2f$.  If
$\mathfrak g_j=\{Y:[h,Y]=jY\}$, then
\[
  \mathfrak g_2=
  \left\{\begin{pmatrix}0&X\\0&0\end{pmatrix}:X\in\Mat_n(F)\right\},
  \qquad \mathfrak g_1=0,
\]
and $\mathfrak g_j=0$ for $j>2$.  Hence the unipotent subgroup appearing in
the generalized Whittaker model attached to the neutral pair $(h,f)$ is
$\exp(\mathfrak g_2)=N$.  For
$Y=\begin{psmallmatrix}0&X\\0&0\end{psmallmatrix}$, its character is
\[
  \psi_f(\exp Y)
  =\psi_0\bigl(\tr(fY)\bigr)
  =\psi_0(\operatorname{Tr}X)
  =\psi(n(X)).
\]
Finally, $f^2=0$ and $\operatorname{rank}(f)=n$, so $f$ has exactly $n$ Jordan
blocks, all of size $2$.
\end{proof}

\begin{theorem}
\label{thm:orbit-criterion}
Let $\pi$ be an irreducible admissible representation of $\GL_{2n}(F)$, and
write
\[
  \lambda(\pi)^\top=(a_1,a_2,\dots)
\]
for the transpose of the partition attached to its wave-front set.  Then
\[
  \pi_{N,\psi}=0
  \quad\Longleftrightarrow\quad
  a_1\geq n+1.
\]
\end{theorem}

\begin{proof}
By Lemmas \ref{lem:closure} and \ref{lem:character},
\[
  \pi_{N,\psi}\neq0
  \quad\Longleftrightarrow\quad
  (2^n)\leq\lambda(\pi).
\]
Transposition reverses dominance, and $(2^n)^\top=(n,n)$, therefore
\[
  (2^n)\leq\lambda(\pi)
  \quad\Longleftrightarrow\quad
  \lambda(\pi)^\top\leq(n,n).
\]
For a partition $\alpha=(a_1,a_2,\dots)$ of $2n$, the inequality
$\alpha\leq(n,n)$ is equivalent to $a_1\leq n$.  This gives
the assertion.
\end{proof}

\section{The M\oe glin--Waldspurger algorithm and the vanishing criterion}

Write $\nu(g)=|\det g|_F$. In this section we recall the M\oe glin--Waldspurger algorithm and deduce from it
an explicit vanishing criterion for the twisted Jacquet module $\pi_{N, \psi}$.

\begin{definition}[Segments and multisegments]
\label{def:segments}
A segment is
\[
  \Delta=[a,b]_\rho
  =\{\nu^a\rho,\nu^{a+1}\rho,\dots,\nu^b\rho\},
\]
where \(d_\rho\) is the integer such that \(\rho\) is an irreducible
supercuspidal representation of \(\GL_{d_\rho}(F)\), and
\(b-a\in\mathbb Z_{\geq0}\).  For any segment
$\Delta=[a,b]_\rho$, write
\[
  b(\Delta)=a,\qquad e(\Delta)=b,\qquad \len(\Delta)=e(\Delta)-b(\Delta)+1.
\]
Thus \(b(\Delta)\) is the left endpoint and \(e(\Delta)\) is the right
endpoint.  Let $\delta(\Delta)$ denote the associated essentially
square-integrable representation.

For segments on the same cuspidal line, set
\[
  \Delta\prec\Delta'
  \quad\Longleftrightarrow\quad
  b(\Delta)<b(\Delta')\leq e(\Delta)+1\leq e(\Delta').
\]
Segments on different cuspidal lines are never related by
$\prec$.

A multisegment is a finite formal sum of segments.  Write
\(\mathfrak m=\Delta_1+\cdots+\Delta_t\) in any order satisfying
\begin{equation}\label{eqn: seg-order}
  i<j \quad\Longrightarrow\quad \Delta_i\not\prec\Delta_j .
\end{equation}
Let $L(\mathfrak m)$ denote the unique irreducible quotient of
$\delta(\Delta_1)\times\cdots\times\delta(\Delta_t)$ which is called a standard module; see
\cite{Bor79} for the Langlands classification.  This quotient is independent
of the chosen order satisfying the displayed condition; see
\cite[Theorem~6.1]{Zel80}.  The
Zelevinsky dual multisegment
$\mathfrak m^{\mathrm t}$ is characterized by
\[
  Z(\mathfrak m)=L(\mathfrak m^{\mathrm t}),
\]
where $Z(\mathfrak m)$ is the Zelevinsky irreducible subrepresentation; see
\cite[Theorem 6.1]{Zel80} and \cite[\S 2]{Deng23}.

\end{definition}

\subsection{The algorithm}

The algorithm is applied separately on each cuspidal line.  Equal segments are
kept as separate rows, and the procedure is repeated after the selected rows
are truncated.

\begin{definition}[M\oe glin--Waldspurger algorithm]
\label{def:MW-algorithm}
Start with a nonempty multisegment on one cuspidal line.
The ordering used above to write the standard module plays no role in the
choices below.
\begin{enumerate}[label=\textup{(\roman*)},leftmargin=2.2em]
  \item Let $d$ be the largest right endpoint.  Among the segments ending at
  $d$, choose $\Delta_1=[a_1,d]$ with $a_1$ maximal.
  \item Having chosen
  $\Delta_j=[a_j,d-j+1]$, choose, if possible, a segment
  $\Delta_{j+1}=[a_{j+1},d-j]$ which linked-precedes \(\Delta_j\), that is,
  satisfies \(\Delta_{j+1}\prec\Delta_j\), with \(a_{j+1}\) maximal.  Stop after
  $r$ selections.
  \item Record the segment $[d-r+1,d]$ in $\mathfrak m^{\mathrm t}$.
  Replace each selected $[a_j,b_j]$ by $[a_j,b_j-1]$, deleting an empty
  segment.
  \item Repeat (i)-(iii) until no segment remains.
\end{enumerate}
We call (i)-(iii) a pass or an MW pass in later sections.
The length of the segment recorded in a pass is exactly the number of segments
selected in that pass.
\end{definition}

If $\alpha=(\alpha_1,\alpha_2,\dots)$ is a partition, we write
$\alpha_1$ also as $(\alpha)_1$ for its first part.

\begin{theorem}
\label{thm:MW-partition}
Let $\pi=L(\mathfrak m)$ be an irreducible representation of $\GL_m(F)$.  For
each segment $\Gamma=[c,d]_\rho$ in $\mathfrak m^{\mathrm t}$, place
$\len(\Gamma)=d-c+1$ into a multiset with multiplicity $d_\rho$.  After
arranging all these integers in decreasing order, one obtains
\[
  \lambda(\pi)^\top.
\]
In particular, 
\[
  \bigl(\lambda(\pi)^\top\bigr)_1
  =\max_{\Gamma\in\mathfrak m^{\mathrm t}}\len(\Gamma).
\]
\end{theorem}

\begin{proof}
Put \(\mathfrak n=\mathfrak m^{\mathrm t}\).  Then
\(\pi=L(\mathfrak m)=Z(\mathfrak n)\).  M\oe glin--Waldspurger compute the
maximal nilpotent orbit for \(Z(\mathfrak n)\) in terms of the segments of
\(\mathfrak n\); see
\cite[Chapitre~II, \S II.2, Proposition and proof, pp.~444--446]{MW87}.
For a single segment \(\Gamma=[c,d]_\rho\), the computation on p.~445 of loc.cit.
gives the rectangular partition
\[
  (d_\rho^{\,\len(\Gamma)}).
\]
For a general multisegment, the proof of loc.cit. obtains the orbit by nilpotent
 induction from  these rectangular orbits.  For
\(\GL_m\), nilpotent induction has the following effect on partitions: the
\(j\)-th part of the induced partition is the sum of the \(j\)-th parts of the
inducing partitions, where a partition contributes \(0\) once it has no
\(j\)-th part; see \cite[Chapitre~I, \S7.1, p.~121]{Spa82}.   Hence, if \(\lambda=\lambda(\pi)\), then
\[
  \lambda_j=\sum_{\Gamma\in\mathfrak n,\ \len(\Gamma)\geq j} d_\rho
  \qquad (j\geq1).
\]
By Lemma~\ref{lem:unique-wf-orbit}, this gives the maximal orbit in the wave-front
set $\mathrm{WF}(\pi)$.  The formula says exactly that \(\lambda^\top\) is obtained by
listing each length \(\len(\Gamma)\), for \(\Gamma\in\mathfrak m^{\mathrm t}\),
with multiplicity \(d_\rho\), and then arranging in decreasing order.
\end{proof}

\begin{definition}
\label{def:rMW}
Let 
$
  r_{\MW}(\mathfrak m)
  =\max_{\Gamma\in\mathfrak m^{\mathrm t}}\len(\Gamma).
$
\end{definition}

\begin{theorem}
\label{thm:corrected}
Let $\pi=L(\mathfrak m)$ be an irreducible admissible representation of
$\GL_{2n}(F)$.  Then
\[
  \pi_{N,\psi}=0
  \quad\Longleftrightarrow\quad
  r_{\MW}(\mathfrak m)\geq n+1.
\]
Equivalently,
\[
  \pi_{N,\psi}=0
  \quad\Longleftrightarrow\quad
  \mathfrak m^{\mathrm t}
  \text{ contains a segment of length at least }n+1.
\]
\end{theorem}

\begin{proof}
By Theorem~\ref{thm:MW-partition}, the first part of
$\lambda(\pi)^\top$ is $r_{\MW}(\mathfrak m)$.  The assertion is therefore
Theorem~\ref{thm:orbit-criterion} expressed in multisegment.
\end{proof}

\begin{theorem}
\label{thm:vanishing-parabolic-subquotient}
Let \(\pi\) be an irreducible admissible representation of
\(\GL_{2n}(F)\).  If \(\pi_{N,\psi}=0\), then there are irreducible
representations \(\pi_1,\pi_2\) of \(\GL_n(F)\), with \(\pi_1\)
one-dimensional, such that \(\pi\) is an irreducible subquotient of the
normalized induction
\[
  \pi_1\times\pi_2
  =
  \Ind_{P_{n,n}}^{\GL_{2n}(F)}(\pi_1\boxtimes\pi_2).
\]
Here \(P_{n,n}\) denotes the standard parabolic subgroup of
\(\GL_{2n}(F)\) with Levi subgroup \(\GL_n(F)\times\GL_n(F)\).
\end{theorem}

\begin{proof}
Write \(\pi=L(\mathfrak m)\).  By Theorem~\ref{thm:corrected},
\(\mathfrak m^{\mathrm t}\) contains a segment
\[
  \Gamma=[a,b]_\rho
\]
of length \(r=b-a+1\geq n+1\).  Since \(\pi\) is a representation of
\(\GL_{2n}(F)\), the degree of \(\Gamma\) is at most \(2n\).  Hence
\(d_\rho=1\), otherwise \(d_\rho r\geq2(n+1)>2n\).  Thus \(\rho\) is a
character of \(F^\times\).

Let
\[
  \Gamma_1=[b-n+1,b]_\rho,\qquad
  \Gamma_2=[a,b-n]_\rho.
\]
Then \(\Gamma_1\) has length \(n\), and \(Z(\Gamma_1)\) is the
one-dimensional representation
\[
  g\longmapsto
  \rho(\det g)\,|\det g|_F^{(2b-n+1)/2}
  \qquad(g\in\GL_n(F)).
\]
 Zelevinsky's
two-segment calculation \cite[\S4.6, Proposition]{Zel80} implies
that \(Z(\Gamma)\) is an irreducible subquotient of
\(Z(\Gamma_1)\times Z(\Gamma_2)\).

Now \(\pi=Z(\mathfrak m^{\mathrm t})\).  Hence \(\pi\) is a composition factor
of an induction of the representation  formed by \(Z(\Delta)\), with \(\Delta\)
running through the segments of \(\mathfrak m^{\mathrm t}\).  For
\(\GL_m\), the semisimplification of such an induction is unchanged by permuting
the factors, a consequence of
\cite[Theorem~1.9]{Zel80}.  We may therefore put
\(Z(\Gamma)\) first.  Since \(Z(\Gamma)\) is a subquotient of
\(Z(\Gamma_1)\times Z(\Gamma_2)\), exactness of normalized parabolic induction
shows that \(\pi\) is a subquotient of
\[
  Z(\Gamma_1)\times I_2,
\]
where \(I_2\) is the normalized induction of \(Z(\Gamma_2)\) together with the
representations \(Z(\Delta)\) attached to the remaining segments of
\(\mathfrak m^{\mathrm t}\).  Note that 
\[
  (r-n)+(2n-r)=n,
\]
so \(I_2\) is a finite-length representation of \(\GL_n(F)\).

Take a composition series
\[
  0=I_2^0\subset I_2^1\subset\cdots\subset I_2^q=I_2 .
\]
Exactness of normalized parabolic induction gives a filtration of
\(Z(\Gamma_1)\times I_2\) whose successive quotients are
\[
  Z(\Gamma_1)\times (I_2^j/I_2^{j-1})
  \qquad(1\leq j\leq q).
\]
The irreducible representation \(\pi\) is a subquotient of
\(Z(\Gamma_1)\times\pi_2\) for some irreducible constituent
\(\pi_2=I_2^j/I_2^{j-1}\) of \(I_2\).  Put \(\pi_1=Z(\Gamma_1)\).  Then
\(\pi\) is a subquotient of \(\pi_1\times\pi_2\), and \(\pi_1\) is
one-dimensional.
\end{proof}

\begin{corollary}
\label{cor:unique-nonzero-parabolic-subquotient}
Let \(\pi_1\) be a one-dimensional representation of \(\GL_n(F)\), let
\(\pi_2\) be an irreducible admissible representation of \(\GL_n(F)\), and put
\[
  I=\pi_1\times\pi_2
  =
  \Ind_{P_{n,n}}^{\GL_{2n}(F)}(\pi_1\boxtimes\pi_2).
\]
Then, as a representation of the diagonal subgroup \(\Delta\GL_n(F)\),
\[
  I_{N,\psi}\simeq \pi_1\otimes\pi_2 
\]
which is irreducible.
Consequently, exactly one
irreducible subquotient of \(I\) has nonzero twisted Jacquet module \(\pi_1\otimes\pi_2\).

More explicitly, write \(\pi_1=Z(\Gamma)\), where \(\Gamma\) is the segment of
length \(n\) attached to the character \(\pi_1\), and 
\(\pi_2=Z(\mathfrak n)=L(\mathfrak n^{\mathrm t})\).  Then the unique
irreducible subquotient of \(I\) with nonzero twisted Jacquet module 
\[
  Z(\Gamma+\mathfrak n).
\]
\end{corollary}

\begin{proof}
We first compute the twisted Jacquet module of \(I\).  Let
\(G=\GL_{2n}(F)\) and \(V=F^n\oplus F^n\).  Let
\[
  N=\left\{
  n(X):=\begin{pmatrix}I_n&X\\0&I_n\end{pmatrix}
  :X\in\Mat_n(F)\right\},
  \qquad
  \psi(n(X))=\psi_0(\operatorname{Tr}X).
\]
Identify \(G/P_{n,n}\) with the Grassmannian of
\(n\)-dimensional subspaces \(W\subset V\).  Let \(\mathcal L_W\) be the
fiber over \(W\) of the bundle realizing the normalized induction \(I\).
Explicitly, choose \(g_W\in G\) with \(W=g_W(F^n\oplus0)\).  If
\(s\in\operatorname{Stab}_G(W)\) and
\[
  g_W^{-1}sg_W=
  \begin{pmatrix}a&*\\0&d\end{pmatrix}\in P_{n,n},
\]
then \(s\) acts on \(\mathcal L_W\) by
\[
  \delta_{P_{n,n}}^{1/2}(g_W^{-1}sg_W)\,
  \pi_1(a)\otimes\pi_2(d).
\]
In particular, if the induced action of \(s\) on \(W\) is unipotent and the
induced action on \(V/W\) is trivial, then it acts trivially on
\(\mathcal L_W\).

By the geometric lemma of Bernstein--Zelevinsky
\cite[Lemma~2.12]{BZ77}, the closed
unions of the \(N\)-stable strata \(r(W)\leq j\) give the filtration of
\(I|_N\), where
\[
  r(W)=\dim\operatorname{pr}_2(W),
  \qquad
  A_W=W\cap(F^n\oplus0).
\]
We identify \(A_W\) with its projection to
the first copy of \(F^n\).
Note that if there exists 
\(h\in\operatorname{Stab}_N(W)\)  which acts trivially on \(\mathcal L_W\) and
\(\psi(h)\neq1\), then the graded piece from the filtration supported on that orbit
contributes nothing to \((N,\psi)\)-coinvariants.  

We first check the boundary defined by \(r(W)<n\).  Then
\(\dim A_W=n-r(W)>0\).  Choose \(X\in\Mat_n(F)\) with image contained in
\(A_W\) and with \(\operatorname{Tr}X\neq0\): for instance, choose
\(a\in A_W\) with a nonzero \(j\)-th coordinate and let \(X\) send the
\(j\)-th basis vector to \(a\) and all other basis vectors to \(0\).  After
replacing \(X\) by a scalar multiple, assume
\(\psi_0(\operatorname{Tr}X)\neq1\).  For every \((u,v)\in W\), one has
\[
  n(X)(u,v)=(u+Xv,v)=(u,v)+(Xv,0),
\]
and \((Xv,0)\in A_W\subset W\).  Thus \(n(X)\) stabilizes \(W\) and acts
trivially on \(\mathcal L_W\), but \(\psi(n(X))\neq1\).  As a consequence,
every \(N\)-orbit in the boundary has zero image in
\((N,\psi)\)-coinvariants.

The open stratum defined by \(r(W)=n\) is the single \(N\)-orbit of
\[
  W_0=0\oplus F^n,
\]
and its stabilizer in \(N\) is trivial.  The contribution of the open stratum is
\[
  C_c^\infty(N)\otimes (V_{\pi_1}\otimes V_{\pi_2}),
\]
with \(N\) acting by left translation on \(C_c^\infty(N)\).  Its
\((N,\psi)\)-coinvariants are \(V_{\pi_1}\otimes V_{\pi_2}\) with the
quotient map given by
\[
  f\otimes v\longmapsto
  \left(\int_N f(n)\psi(n)\,dn\right)v .
\]
The diagonal subgroup \(\Delta\GL_n(F)\) preserves both \(\psi\) and the Haar
measure on \(N\).  For \(g\in\GL_n(F)\), the element
\(\operatorname{diag}(g,g)\) fixes \(W_0\) with the fiber action \(\pi_1(g)\otimes\pi_2(g)\).  Since the
boundary contribution is zero, this gives
\[
  I_{N,\psi}\simeq \pi_1\otimes\pi_2
\]
as a \(\Delta\GL_n(F)\)-module which is
irreducible.

The functor \((-)_{N,\psi}\) is exact.  Applying it to a composition series of
\(I\) gives a filtration of the irreducible module \(I_{N,\psi}\).  Therefore
exactly one irreducible subquotient of \(I\) has nonzero twisted Jacquet
module \(I_{N,\psi}\simeq\pi_1\otimes\pi_2\).

It remains to identify the subquotient.  By Zelevinsky's theorem on
composition factors \cite[Theorem 7.1]{Zel80}, \(Z(\Gamma+\mathfrak n)\)
occurs as an irreducible subquotient of
\[
  Z(\Gamma)\times Z(\mathfrak n)=\pi_1\times\pi_2.
\]
The segment \(\Gamma\) has length \(n\), and every segment of \(\mathfrak n\)
has length at most \(n\), since \(\mathfrak n\) has total degree \(n\).  Hence
the multisegment \(\Gamma+\mathfrak n\) contains no segment of length
\(n+1\).  As
\[
  Z(\Gamma+\mathfrak n)=L\bigl((\Gamma+\mathfrak n)^{\mathrm t}\bigr),
\]
Theorem~\ref{thm:corrected} shows that
\[
  Z(\Gamma+\mathfrak n)_{N,\psi}\neq0.
\]
By the uniqueness just proved, this is the desired subquotient.
\end{proof}

\begin{remark}
The first assertion of
Corollary~\ref{cor:unique-nonzero-parabolic-subquotient}, namely the
calculation
\[
  (\pi_1\times\pi_2)_{N,\psi}\simeq\pi_1\otimes\pi_2,
\]
is also contained in \cite[Proposition~5.1(2)]{HV26}.
\end{remark}

\section{Rankin--Selberg factors and Prasad's conjecture}
\label{sec:L-functions}

Sections 1 and 2 give complete criteria in terms of the wave-front
set and the Zelevinsky classification.  We now collect the $L$-function
preliminaries needed to compare those criteria with the original Langlands
parameter.  We also spell out the implication in Prasad's conjecture.

Fix once and for all one base supercuspidal representation on each cuspidal line and absorb
all integral powers of $\nu$ into the endpoints. 

The following propositions relate the Zelevinsky notion of "linked segments" in terms of the existence of poles of  Rankin-Selberg 
$L$ functions, which we have not seen in the literature.
\begin{proposition}
\label{prop:RS}
Let
$\Delta=[a,b]_\rho$ and $\Delta'=[a',b']_{\rho'}$, with lengths
$\len=b-a+1$ and $\len'=b'-a'+1$.   Then
\begin{equation}
\label{eq:RS}
  L\bigl(s,\delta(\Delta)\times\delta(\Delta')^\vee\bigr)
  =\prod_{k=0}^{\min(\len,\len')-1}
  L\bigl(s+b-a'-k,\rho\times\rho'^\vee\bigr).
\end{equation}
Consequently, this factor has a pole at $s=1$ if and only if
$\Delta\prec\Delta'$.
\end{proposition}

\begin{proof}
The local Langlands parameter of $\delta([a,b]_\rho)$ is
\[
  \nu^{(a+b)/2}\phi_\rho\otimes S_{\len},
\]
where $\phi_\rho$ is the Weil--Deligne parameter of $\rho$ under the local
Langlands correspondence and $S_r$ is the $r$-dimensional irreducible
algebraic representation of $\mathrm{SL}_2(\mathbb C)$.  The
Clebsch--Gordan decomposition
\[
  S_{\len}\otimes S_{\len'}
  \simeq
  \bigoplus_{k=0}^{\min(\len,\len')-1}
  S_{\len+\len'-1-2k}
\]
gives \eqref{eq:RS}; the resulting shift is
\[
  \frac{a+b-a'-b'}2
  +\frac{\len+\len'-2-2k}{2}
  =b-a'-k.
\]
If $\rho$ and $\rho'$ are on different cuspidal lines, no factor in
\eqref{eq:RS} has a pole at $s=1$.  If
$\rho=\rho'$, then $L(s,\rho\times\rho^\vee)$ has a simple pole at $s=0$.
Thus a pole at $s=1$ occurs exactly when
\[
  k=b-a'+1
\]
for an integer $k$ satisfying
$0\leq k\leq\min(\len,\len')-1$.  These inequalities are equivalent to
\[
  a<a'\leq b+1\leq b',
\]
which is precisely $\Delta\prec\Delta'$.
\end{proof}

\begin{proposition}
\label{prop:positive-pole-pairs}
Let \(\Delta=[a,b]_\rho\) and \(\Delta'=[a',b']_{\rho'}\) be two segments, with
lengths \(\len=b-a+1\) and \(\len'=b'-a'+1\).  
\begin{itemize}
\item If the two segments are on the
same cuspidal line, then for a
positive integer \(t\), the pair \((\Delta,\Delta')\), taken in the order of
the two variables in the following \(L\)-factor, contributes one pole to
\(L(s,\delta(\Delta)\times\delta(\Delta')^\vee)\) at \(s=t\) precisely
when
\begin{equation}
\label{eq:pair-pole-t}
  0\leq t+b-a'\leq\min(\len,\len')-1.
\end{equation}
\item If the two segments are on different cuspidal lines, the pair contributes
no pole at any positive integer.
\end{itemize}

Consequently, for
\(\pi=L(\mathfrak m)\) with
\[
  \mathfrak m=\Delta_1+\cdots+\Delta_q,
  \qquad
  \Delta_i=[a_i,b_i]_{\rho_i},
  \qquad
  \ell_i=b_i-a_i+1,
\]
we have 
\begin{equation}
\label{eq:self-rs-multisegment}
  L(s,\pi\times\pi^\vee)
  =
  \prod_{i,j=1}^q
  L\bigl(s,\delta(\Delta_i)\times\delta(\Delta_j)^\vee\bigr)
  =
  \prod_{i,j=1}^q
  \prod_{k=0}^{\min(\ell_i,\ell_j)-1}
  L\bigl(s+b_i-a_j-k,\rho_i\times\rho_j^\vee\bigr).
\end{equation}
Only pairs on the same cuspidal line contribute poles at positive integers.
\end{proposition}

\begin{proof}
This is the same calculation as in Proposition~\ref{prop:RS}.  By
\eqref{eq:RS}, if the two segments are on different cuspidal lines, no factor
has a pole at a positive integer.  On the same cuspidal line, a pole at \(s=t\)
can occur only from a factor
\[
  L(s+b-a'-k,\rho\times\rho^\vee)
\]
with \(0\leq k\leq\min(\len,\len')-1\).  Since
\(L(s,\rho\times\rho^\vee)\) has a simple pole at \(s=0\), this is equivalent
to
\[
  k=t+b-a'.
\]
The required condition is exactly that this integer \(k\) lie in the displayed
range.  Formula \eqref{eq:self-rs-multisegment} follows by multiplicativity of
local \(L\)-functions and Proposition~\ref{prop:RS}.
\end{proof}

\begin{conjecture}\cite[Conjecture 6.3]{HV26}
\label{conj:prasad}
Let $\pi$ be an irreducible smooth representation of $\GL_{2n}(F)$.  Then
\[
  \pi_{N,\psi}=0
\]
if and only if
\[
  \pord_{s=t}L(s,\pi\times\pi^\vee)\geq n-t+1
  \qquad(1\leq t\leq n).
\]
\end{conjecture}

\subsection{The vanishing-to-poles implication}

We first isolate the combinatorial feature of the M\oe glin--Waldspurger
algorithm that forces the poles.  By Proposition \ref{prop:positive-pole-pairs}, we can and will assume that all segments in the next statements are on
one fixed cuspidal line.  

\begin{definition}
Let $c\in\mathbb Z$.  A sequence of segments
\[
  \Delta_i=[a_i,b_i]_\rho,
  \qquad 0\leq i\leq r-1,
\]
is called a \emph{$c$-diagonal ladder} if
\[
  a_0<a_1<\cdots<a_{r-1},
  \qquad
  b_0<b_1<\cdots<b_{r-1},
\]
and
\[
  c+i\in\Delta_i
  \qquad(0\leq i\leq r-1).
\]
\end{definition}

\begin{lemma}[Lifting through one MW pass]
\label{lem:MW-lifting}
Let $\mathfrak a^\dagger$ be obtained from a multisegment $\mathfrak a$ by one
pass of the M\oe glin--Waldspurger algorithm in Definition
\ref{def:MW-algorithm}.  Every $c$-diagonal ladder in
$\mathfrak a^\dagger$ can be lifted to a $c$-diagonal ladder of the same
length in $\mathfrak a$.
\end{lemma}

\begin{proof}
Write the segments selected in this pass, in their order of selection, as
\[
  S_j=[u_j,D-j]_\rho,
  \qquad 0\leq j\leq k-1.
\]
The selection rule gives
\[
  u_0>u_1>\cdots>u_{k-1}.
\]
After the pass, $S_j$ is replaced by
\[
  S_j^-=[u_j,D-j-1]_\rho,
\]
with an empty segment deleted.  Let
$\Delta_0,\ldots,\Delta_{r-1}$ be a $c$-diagonal ladder in
$\mathfrak a^\dagger$.  Lift each segment to its ancestor in
$\mathfrak a$.  Left endpoints and the conditions
$c+i\in\Delta_i$ are unchanged.  A failure of strict increase among the right
endpoints can occur only in the following configuration:
\[
  \Delta_i=S_j^-,\qquad
  \Delta_{i+1}=U,
\]
where $U$ was not selected and
\[
  e(U)=e(S_j^-)+1=D-j.
\]
where $e(U)$ is the right endpoint of $U$. Since the ladder has strictly increasing left
endpoints,
\[
  u_j<b(U),
\]
where $b(U)$ is the left endpoint of $U$.

If $j=0$, then before truncation both $S_0$ and $U$ ended at $D$, contradicting
the choice of $S_0$ with maximal left endpoint among the segments ending at
$D$.  Thus $j\geq1$.  We claim that
\[
  u_{j-1}\leq b(U).
\]
Indeed, if $b(U)<u_{j-1}$, then before truncation $U$ was eligible at the step
at which $S_j$ was chosen after $S_{j-1}$.  The maximality condition in the MW
algorithm would give $u_j\geq b(U)$, contrary to $u_j<b(U)$.

Replace $U$ in the ladder by
\[
  S_{j-1}^-=[u_{j-1},D-j]_\rho.
\]
This replacement has the same right endpoint as $U$, its left endpoint lies
strictly to the right of $u_j$ and weakly to the left of $b(U)$, and it still
contains $c+i+1$.  Hence it preserves the $c$-diagonal ladder in
$\mathfrak a^\dagger$.  Moreover, after lifting, the two consecutive terms
$S_j^-,S_{j-1}^-$ acquire the strictly increasing right endpoints
$D-j,D-j+1$.

The replacement can create the same issue only with the next term of the
ladder.  Repeating the repair moves successively from $S_j^-$ to
$S_{j-1}^-,S_{j-2}^-,\ldots$.  The process terminates, and it can never pass
$S_0^-$ by the first-step maximality argument above.  Scanning the ladder from
left to right therefore removes every possible failure.  The resulting
ancestors in $\mathfrak a$ form a $c$-diagonal ladder of the same length.
\end{proof}

\begin{lemma}
\label{lem:dual-segment-diagonal-ladder}
Let $\mathfrak m$ be a multisegment.  If
\[
  [c,d]_\rho
\]
is a segment in $\mathfrak m^{\mathrm t}$ and $r=d-c+1$, then the $\rho$-part of
$\mathfrak m$ contains a $c$-diagonal ladder
\[
  \Delta_0,\Delta_1,\ldots,\Delta_{r-1}
\]
of length $r$.
\end{lemma}

\begin{proof}
Run the MW algorithm on the $\rho$-part of $\mathfrak m$.  Consider the pass
that records the segment $[c,d]_\rho$.  Immediately before that pass, the
selected segments, listed in their order of selection, have the form
\[
  S_0=[u_0,d]_\rho,
  S_1=[u_1,d-1]_\rho,
  \ldots,
  S_{r-1}=[u_{r-1},c]_\rho,
\]
with
\[
  u_0>u_1>\cdots>u_{r-1}.
\]
Consequently,
\[
  \Delta_i=S_{r-1-i},
  \qquad 0\leq i\leq r-1,
\]
is a $c$-diagonal ladder in the current multisegment: its right endpoints are
$c,c+1,\ldots,d$, its left endpoints are strictly increasing, and
$c+i\in\Delta_i$.

The current multisegment was obtained from the original $\mathfrak m$ by the
preceding MW passes.  Apply Lemma~\ref{lem:MW-lifting} successively in reverse
order through those passes.  This produces a $c$-diagonal ladder of length
$r$ in $\mathfrak m$ itself.
\end{proof}

For later use, note that condition \eqref{eq:pair-pole-t} is equivalently
\begin{equation}
\label{eq:pair-pole-t-endpoints}
  a+t\leq a'\leq b+t\leq b'.
\end{equation}
Indeed, its three inequalities are respectively the upper bound coming from
$\len-1$, the nonnegativity condition, and the upper bound coming from
$\len'-1$.

\begin{theorem}
\label{thm:general-vanishing-to-poles}
Let $\pi=L(\mathfrak m)$ be an irreducible representation of $\GL_M(F)$ and
put
\[
  r=r_{\MW}(\mathfrak m).
\]
Then, for every positive integer $t$,
\begin{equation}
\label{eq:general-pole-lower-bound}
  \pord_{s=t}L(s,\pi\times\pi^\vee)\geq \max\{r-t,0\}.
\end{equation}
In particular, if $M=2n$, then
\[
  \pi_{N,\psi}=0
  \quad\Longrightarrow\quad
  \pord_{s=t}L(s,\pi\times\pi^\vee)\geq n-t+1
  \qquad(1\leq t\leq n).
\]
Thus the vanishing-to-poles implication in Conjecture
\ref{conj:prasad} holds for every irreducible representation of
$\GL_{2n}(F)$.
\end{theorem}

\begin{proof}
Choose a segment $[c,d]_\rho$ of length $r$ in
$\mathfrak m^{\mathrm t}$.  By Lemma
\ref{lem:dual-segment-diagonal-ladder}, the original multisegment contains a
$c$-diagonal ladder
\[
  \Delta_i=[a_i,b_i]_\rho,
  \qquad 0\leq i\leq r-1.
\]
If \(t\geq r\), then the right hand side of
\eqref{eq:general-pole-lower-bound} is zero, so there is nothing to prove.
Assume \(1\leq t<r\).  Fix \(i\) with \(0\leq i\leq r-t-1\).  Since the
endpoints in the ladder are strictly increasing integers,
\[
  a_i+t\leq a_{i+t},
  \qquad
  b_i+t\leq b_{i+t}.
\]
The diagonal conditions \(c+i\in\Delta_i\) and
\(c+i+t\in\Delta_{i+t}\) give
\[
  a_{i+t}\leq c+i+t
  \qquad\text{and}\qquad
  c+i+t\leq b_i+t.
\]
Combining these inequalities, we obtain
\[
  a_i+t\leq a_{i+t}\leq b_i+t\leq b_{i+t}.
\]
By \eqref{eq:pair-pole-t-endpoints}, the pair
$(\Delta_i,\Delta_{i+t})$, taken in this order,  contributes one pole at $s=t$ to
the adjoint L factor $L(s,\pi\times\pi^\vee)$ as written in  \eqref{eq:self-rs-multisegment}.  As \(i\) varies, these \(r-t\) ordered pairs are
distinct pairs.  Proposition
\ref{prop:positive-pole-pairs} and multiplicativity give one pole for each of
these factors, and local \(L\)-factors have no zeros that could cancel a pole.
Hence
\[
  \pord_{s=t}L(s,\pi\times\pi^\vee)\geq r-t.
\]
This proves \eqref{eq:general-pole-lower-bound}.

Now suppose $M=2n$ and $\pi_{N,\psi}=0$.  Theorem
\ref{thm:corrected} gives $r\geq n+1$.  Therefore, for $1\leq t\leq n$,
\[
  \pord_{s=t}L(s,\pi\times\pi^\vee)
  \geq r-t\geq n-t+1,
\]
as required.
\end{proof}

\begin{remark}
The condition
\[
  \pord_{s=t}L(s,\pi\times\pi^\vee)\geq \max\{n-t+1,0\}
  \qquad(t\geq1)
\]
does not repair the converse implication.  
\end{remark}



\section{The principal series containing the trivial representation}
\label{sec:trivial-principal-series}

Let $m\geq1$ and let $\rho$ be a smooth character of $F^\times$.  Consider the
normalized principal series
\begin{equation}
\label{eq:sigma-m-rho}
  \Sigma_m(\rho)
  :=\nu^{m-1}\rho\times\nu^{m-2}\rho\times\cdots\times\rho
\end{equation}
of $\GL_m(F)$.  For $\rho=\nu^{-(m-1)/2}$, this is the principal
series containing the trivial representation.

For an ordered partition \(\alpha=(\alpha_1,\dots,\alpha_k)\) of \(m\), put
\[
  s_0=0,\qquad s_j=\alpha_1+\cdots+\alpha_j\quad(1\leq j\leq k).
\]
Its break set is
\[
  B(\alpha)=\{s_1,\dots,s_{k-1}\}\subseteq\{1,\dots,m-1\}.
\]
Thus \(0=s_0<s_1<\cdots<s_k=m\) is obtained from \(B(\alpha)\) by adjoining
\(0\) and \(m\).  Attach the multisegment
\begin{equation}
\label{eq:p-alpha}
  \mathfrak p_\alpha
  :=\sum_{j=1}^k[s_{j-1},s_j-1]_\rho.
\end{equation}
Let  
\(L(\mathfrak p_\alpha)\) be the Langlands quotient formed from the multisegment   $\mathfrak p_\alpha$.
By Zelevinsky's classification \cite{Zel80}, the distinct irreducible
subquotients of $\Sigma_m(\rho)$ are exactly
\begin{equation}
\label{eq:pi-alpha}
  \pi_\alpha:=L(\mathfrak p_\alpha),
\end{equation}
as $\alpha$ ranges over the $2^{m-1}$ ordered partitions of $m$.

\subsection{The Langlands parameter}

\begin{proposition}
\label{prop:parameter-principal-series}
The Langlands parameter of $\pi_\alpha$ is
\[
  \phi_{\pi_\alpha}
  :=\bigoplus_{j=1}^k
  \rho\nu^{(s_{j-1}+s_j-1)/2}\otimes S_{\alpha_j}.
\]
For $\rho=\nu^{-(m-1)/2}$ this becomes
\[
  \phi_{\pi_\alpha}
  =\bigoplus_{j=1}^k
  \nu^{(s_{j-1}+s_j-m)/2}\otimes S_{\alpha_j}.
\]
\end{proposition}

\begin{proof}
The segment $[s_{j-1},s_j-1]_\rho$ contributes
\[
  \rho\nu^{(s_{j-1}+s_j-1)/2}\otimes S_{s_j-s_{j-1}}
  =\rho\nu^{(s_{j-1}+s_j-1)/2}\otimes S_{\alpha_j}.
\]
Taking the direct sum over the segments gives the first formula;
the second follows by substituting $\rho=\nu^{-(m-1)/2}$.
\end{proof}

\subsection{Parabolic indexing and complementary ordered partitions}

For an ordered partition \(\gamma\) of \(m\), let \(B(\gamma)\) be its break
set, as above.  Define the complementary ordered partition $\gamma^\complement$
by
\begin{equation}
\label{eq:complementary-ordered-partition}
  B(\gamma^\complement)
  =\{1,\dots,m-1\}\smallsetminus B(\gamma).
\end{equation}
The operation $\gamma\mapsto\gamma^\complement$ is an involution.  It
exchanges $(m)$ and $(1,\dots,1)$.
We write \(\gamma^\downarrow\) for the partition obtained by rearranging the
parts of \(\gamma\) in decreasing order.

Let $P_\beta$ be the standard parabolic subgroup with Levi size of the ordered partition
$\beta$, and write
\[
  v_{P_\beta}^G
  =C^\infty(G/P_\beta)
   \Big/
   \sum_{Q\supsetneq P_\beta}C^\infty(G/Q)
\]
for the corresponding generalized Steinberg representation.

\begin{proposition}
\label{prop:parabolic-indexing}
For $\rho=\nu^{-(m-1)/2}$ and every ordered partition $\beta$ of $m$,
\[
  v_{P_\beta}^G\simeq\pi_{\beta^\complement}.
\]
\end{proposition}

\begin{proof}
In terms of subsets of the simple roots \(\{1,\dots,m-1\}\), the parabolic
\(P_\beta\) has Levi size  given by the ordered partition $\beta$.  The simple roots in its Levi are
\[
  I_\beta=\{1,\dots,m-1\}\smallsetminus B(\beta).
\]
Let
\[
  0=s_0<s_1<\cdots<s_r=m
\]
be obtained by adjoining \(0\) and \(m\) to \(B(\beta)\).  
For any parabolic $P$, let $\delta_P$ be the associated modulus character.
For
\(\rho=\nu^{-(m-1)/2}\), the character
\(\delta_{P_\beta}^{-1/2}\) on the block
\(\GL_{s_\ell-s_{\ell-1}}(F)\) is
\[
  g\longmapsto
  \rho(\det g)|\det g|_F^{(s_{\ell-1}+s_\ell-1)/2},
  \qquad g\in\GL_{s_\ell-s_{\ell-1}}(F).
\]
This is the one-dimensional representation
\(Z([s_{\ell-1},s_\ell-1]_\rho)\).
Thus \(C^\infty(G/P_\beta)\) can be written as the
normalized induction
\[
  C^\infty(G/P_\beta)\simeq
  \Ind_{P_\beta}^G(\delta_{P_\beta}^{-1/2})
  \simeq
  Z([s_0,s_1-1]_\rho)\times\cdots\times
  Z([s_{r-1},s_r-1]_\rho).
\]
Moreover, under the pullback inclusions \(C^\infty(G/Q)\subset
C^\infty(G/P_\beta)\) for \(Q\supset P_\beta\), the quotient defining
\(v_{P_\beta}^G\) is
\[
  v_{P_\beta}^G\simeq
  \Ind_{P_\beta}^G(\delta_{P_\beta}^{-1/2})
  \Big/
  \sum_{Q\supsetneq P_\beta}
  \Ind_Q^G(\delta_Q^{-1/2}).
\]
Zelevinsky identifies this generalized Steinberg quotient with the
irreducible representation \(Z(\mathfrak m)\) attached to the multisegment
determined by the blocks of \(P_\beta\):
\[
  v_{P_\beta}^G\simeq
  Z\left(\sum_{\ell=1}^r [s_{\ell-1},s_\ell-1]_\rho\right);
\]
see \cite[\S2.2, Proposition~2.10]{Zel80}.
By Zelevinsky duality, \(Z(\mathfrak m)=L(\mathfrak m^{\mathrm t})\), and
 Proposition~\ref{prop:diagonal-complement}
below gives the dual multisegment $\mathfrak m^t$ in this case.  Thus, if
\[
  0=t_0<t_1<\cdots<t_k=m
\]
is obtained by adjoining \(0\) and \(m\) to \(I_\beta\), then
\[
  \left(\sum_{\ell=1}^r [s_{\ell-1},s_\ell-1]_\rho\right)^{\mathrm t}
  =
  \sum_{j=1}^k [t_{j-1},t_j-1]_\rho,
\]
and therefore
\[
  v_{P_\beta}^G\simeq
  L\left(\sum_{j=1}^k [t_{j-1},t_j-1]_\rho\right).
\]
If \(\alpha=(t_1-t_0,\dots,t_k-t_{k-1})\), then the corresponding multisegment is
\(\mathfrak p_\alpha\) from \eqref{eq:p-alpha}, and
\[
  B(\alpha)=I_\beta
  =\{1,\dots,m-1\}\smallsetminus B(\beta),
\]
so \(\alpha=\beta^\complement\).  The extreme cases give
$v_G^G=\one=\pi_{(1,\dots,1)}$ and
$v_B^G=\operatorname{St}_m=\pi_{(m)}$.
\end{proof}

It follows directly that 
\begin{corollary}
The parts occurring as dimensions of the Deligne
$\mathrm{SL}_2$-summands in the Langlands parameter of $v_{P_\beta}^G$ are the parts of
the complementary ordered partition $\beta^\complement$.
\end{corollary}

\subsection{The Zelevinsky dual and the wave-front set}

For an integer $t$, given an ordered partition
$\alpha=(\alpha_1,\dots,\alpha_k)$ of $m$, let
\[
  0=s_0<s_1<\cdots<s_k=m
\]
be obtained from \(B(\alpha)\) by adjoining \(0\) and \(m\).
Define
\begin{equation}
\label{eq:h-alpha}
  h_\alpha(t)
  =\#\left\{
    1\leq j\leq k:
    s_{j-1}-j\leq t\leq s_j-1-j
  \right\}.
\end{equation}
If one draws the segment $[s_{j-1},s_j-1]_\rho$ in row $j$, with columns
\(s_{j-1}\leq c\leq s_j-1\), then $h_\alpha(t)$ is the cardinality of the
intersection of this diagram with the diagonal $c-j=t$ (see the picture after Proposition
\ref{prop:trivial-principal-series}).

\begin{proposition}
\label{prop:trivial-principal-series}
The Zelevinsky dual of $\mathfrak p_\alpha$ is
\[
  \mathfrak p_\alpha^{\mathrm t}
  =\sum_{t:\,h_\alpha(t)>0}
  [t+j_-(t),t+j_+(t)]_\rho,
\]
where
\[
  j_-(t)=\min\{j:s_{j-1}-j\leq t\leq s_j-1-j\},
\]
\[
  j_+(t)=\max\{j:s_{j-1}-j\leq t\leq s_j-1-j\}.
\]
In particular,
\[
  r_{\MW}(\mathfrak p_\alpha)=\max_t h_\alpha(t).
\]
\end{proposition}

\begin{proof}
The shifted intervals
\[
  I_j=[s_{j-1}-j,s_j-1-j]\qquad(1\leq j\leq k)
\]
meet consecutively because, for \(1\leq j<k\), the right endpoint of \(I_j\)
equals the left endpoint of \(I_{j+1}\).
Hence, for fixed $t$, the rows counted by $h_\alpha(t)$ form an interval of
row indices.  In the M\oe glin--Waldspurger algorithm, the pass associated
with the diagonal $c-j=t$ selects exactly these rows, from right to left, and
records the segment whose endpoints are the first and last column numbers on
that diagonal, namely $t+j_-(t)$ and $t+j_+(t)$.  Repeating over all diagonals
gives the formula.  Its recorded segment has length
$j_+(t)-j_-(t)+1=h_\alpha(t)$.
\end{proof}

\medskip
\noindent\textit{Picture.}
Take $\alpha=(1,1,2,1,1)$.  Then
\[
  \mathfrak p_\alpha
  =[0,0]_\rho+[1,1]_\rho+[2,3]_\rho+[4,4]_\rho+[5,5]_\rho.
\]
Write these segments, in the displayed order, as
\(\Delta_1,\dots,\Delta_5\).  In the diagram below, \(\Delta_j\) is the row
indexed by \(j\), drawn at vertical coordinate \(1-j\), and the horizontal
coordinate is the original exponent \(c\).  The colored dots mark the shifted
diagonals \(t=c-j\).  The M\oe
glin--Waldspurger algorithm records one segment from each colored diagonal.
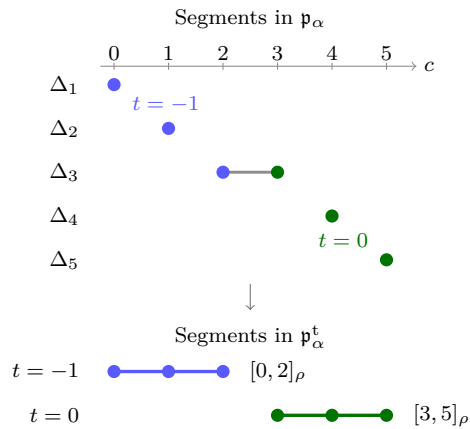
\begin{center}
\begin{tikzpicture}[
  x=.72cm,
  y=.58cm,
  every node/.style={font=\scriptsize},
  seg/.style={draw=black!45, line width=1.1pt, line cap=round},
  blueseg/.style={draw=blue!65, line width=1.25pt, line cap=round},
  greenseg/.style={draw=green!45!black, line width=1.25pt, line cap=round},
  dot/.style={circle, fill=black!45, inner sep=1.45pt},
  bluedot/.style={circle, fill=blue!65, inner sep=1.75pt},
  greendot/.style={circle, fill=green!45!black, inner sep=1.75pt},
  label/.style={font=\scriptsize}
]
  \node[label] at (2.5,1.5) { Segments in  $\mathfrak p_\alpha$};
  \draw[->, black!45] (-.25,.42) -- (5.55,.42);
  \foreach \c in {0,...,5} {
    \draw[black!45] (\c,.34) -- (\c,.50);
    \node[label] at (\c,.72) {\(\c\)};
  }
  \node[label] at (5.78,.42) {\(c\)};

  \foreach \j/\y in {1/0,2/-1,3/-2,4/-3,5/-4} {
    \node[label,left] at (-.45,\y) {\(\Delta_{\j}\)};
  }

  \draw[seg] (2,-2) -- (3,-2);

  \node[bluedot] at (0,0) {};
  \node[bluedot] at (1,-1) {};
  \node[bluedot] at (2,-2) {};
  \node[greendot] at (3,-2) {};
  \node[greendot] at (4,-3) {};
  \node[greendot] at (5,-4) {};
  \node[blue!65, label,right] at (.15,-.45) {\(t=-1\)};
  \node[green!45!black, label,left] at (4.85,-3.55) {\(t=0\)};

  \draw[->, black!55] (2.5,-4.55) -- (2.5,-5.15);
  \node[label,right] at (2.62,-4.85) {};

  \node[label] at (2.5,-5.75) {Segments in \(\mathfrak p_\alpha^{\mathrm t}\)};
  \node[label,left] at (-.45,-6.55) {\(t=-1\)};
  \draw[blueseg] (0,-6.55) -- (2,-6.55);
  \foreach \x in {0,1,2} {
    \node[bluedot] at (\x,-6.55) {};
  }
  \node[label,right] at (2.3,-6.55) {\([0,2]_\rho\)};

  \node[label,left] at (-.45,-7.55) {\(t=0\)};
  \draw[greenseg] (3,-7.55) -- (5,-7.55);
  \foreach \x in {3,4,5} {
    \node[greendot] at (\x,-7.55) {};
  }
  \node[label,right] at (5.3,-7.55) {\([3,5]_\rho\)};
\end{tikzpicture}
\end{center}
Hence
\[
  \mathfrak p_\alpha^{\mathrm t}=[0,2]_\rho+[3,5]_\rho,
  \qquad
  h_\alpha(-1)=h_\alpha(0)=3.
\]

\begin{proposition}
\label{prop:diagonal-complement}
Let $\beta=\alpha^\complement$.  Then, in their natural order,
\[
  \bigl(h_\alpha(-1),h_\alpha(0),\dots,h_\alpha(m-k-1)\bigr)
  =\beta.
\]
Consequently, the multiset of lengths of the segments in
$\mathfrak p_\alpha^{\mathrm t}$ is the multiset of parts of $\beta$.
\end{proposition}

\begin{proof}
For \(r=0,1,\dots,m-1\), let \(j(r)\) be the unique index such that
\[
  s_{j(r)-1}\leq r\leq s_{j(r)}-1.
\]
Thus \(r\) is the column coordinate of an entry in row \(j(r)\).  Its
diagonal is
\[
  d(r)=r-j(r).
\]
When one passes from $r-1$ to $r$, the diagonal remains unchanged exactly
when $r\in B(\alpha)$, because in that case \(r\) is the first column of the
next row;
otherwise \(r-1\) and \(r\) lie in the same row, so the diagonal increases by
one, as illustrated in the preceding picture.  Consider the maximal
consecutive intervals of \(r\) on which \(d(r)\) is constant.  Their lengths
form an ordered partition with break set
\[
  \{1,\dots,m-1\}\smallsetminus B(\alpha).
\]
This ordered partition is
$\alpha^\complement=\beta$, while those run lengths are exactly the diagonal
counts $h_\alpha(-1),h_\alpha(0),\dots,h_\alpha(m-k-1)$.
\end{proof}

\begin{corollary}[Wave-front set of the generalized Steinberg quotient]
\label{cor:wf-principal-series}
Let $\lambda_\alpha$ be the partition attached to the wave-front set of
$\pi_\alpha=L(\mathfrak p_\alpha)$ and put
$\beta=\alpha^\complement$.  Then
\[
  \lambda_\alpha^\top=\beta^\downarrow,
  \qquad
  \lambda_\alpha=(\beta^\downarrow)^\top.
\]
In particular, for the generalized Steinberg representation attached to
$P_\beta$,
\[
  \operatorname{WF}(v_{P_\beta}^G)
  =\overline{\mathcal O_{(\beta^\downarrow)^\top}}^{\,\mathrm{an}},
\]
the analytic closure of the Richardson orbit of $P_\beta$.
\end{corollary}

\begin{proof}
By Propositions~\ref{prop:trivial-principal-series} and
\ref{prop:diagonal-complement}, the segment lengths in
$\mathfrak p_\alpha^{\mathrm t}$ are the parts of $\beta$.  Theorem
\ref{thm:MW-partition} therefore gives
$\lambda_\alpha^\top=\beta^\downarrow$.  The final assertion follows from
Proposition~\ref{prop:parabolic-indexing} and the usual formula for the
Richardson partition of a parabolic in $\GL_m$.
\end{proof}

\begin{remark}
Different ordered standard parabolics \(P_\beta\) give distinct generalized
Steinberg representations \(v_{P_\beta}^G\).  Their wave-front sets
depend only on the associate class of \(P_\beta\).
\end{remark}

\subsection{Vanishing and pole counts}

Suppose now that $m=2n$.  The wave-front set criterion immediately gives the
following exact answer.

\begin{corollary}
\label{cor:trivial-principal-series-tjm}
For the constituent $\pi_\alpha=L(\mathfrak p_\alpha)$ of
$\Sigma_{2n}(\rho)$,
\[
  (\pi_\alpha)_{N,\psi}=0
  \quad\Longleftrightarrow\quad
  \max_t h_\alpha(t)\geq n+1.
\]
\end{corollary}

For $\rho=\nu^{-(2n-1)/2}$, the trivial representation corresponds to
$\alpha=(1,1,\dots,1)$; then $h_\alpha(-1)=2n$, so its twisted Jacquet module
is zero.

\begin{proof}
Combine Proposition~\ref{prop:trivial-principal-series} with
Theorem~\ref{thm:corrected}.
\end{proof}

For a positive integer $t$ and an ordered partition
$\alpha=(\alpha_1,\dots,\alpha_k)$ of $2n$, let
\[
  0=s_0<s_1<\cdots<s_k=2n
\]
be obtained from \(B(\alpha)\) by adjoining \(0\) and \(2n\).  Define
\[
  c_\alpha(t)
  =\#\left\{
  (j,q,u):
  \begin{array}{l}
  1\leq j<q\leq k,\\
  0\leq u<\min(\alpha_j,\alpha_q),\\
  t=s_{q-1}-s_j+1+u
  \end{array}
  \right\}.
\]

\begin{proposition}
\label{prop:principal-series-poles}
For $\pi_\alpha=L(\mathfrak p_\alpha)$,
\[
  \pord_{s=t}L(s,\pi_\alpha\times\pi_\alpha^\vee)=c_\alpha(t)
  \qquad(t\geq1).
\]
\end{proposition}

\begin{proof}
Apply Proposition~\ref{prop:RS} to the two rows
$[s_{j-1},s_j-1]_\rho$ and $[s_{q-1},s_q-1]_\rho$, in this order in the
\(L\)-factor \eqref{eq:self-rs-multisegment}.  Pairs with $q\leq j$
give no pole at a positive integer.  For $j<q$, formula
\eqref{eq:pair-pole-t} says that the possible poles are precisely
\[
  s_{q-1}-s_j+1+u,
  \qquad
  0\leq u<\min(\alpha_j,\alpha_q).
\]
Counting the pairs with respect to this order gives $c_\alpha(t)$.  
\end{proof}

The above proposition allows us to show that the implication from the pole inequalities to vanishing
 of twisted Jacquet module in Conjecture~\ref{conj:prasad} is false.  The following family gives counterexamples in every
even rank at least six.

\begin{corollary}
\label{cor:principal-series-pole-counterexamples}
Let $n\geq3$ and let
\[
  \alpha_n=(\underbrace{1,\dots,1}_{n-1},2,
  \underbrace{1,\dots,1}_{n-1})
\]
be an ordered partition of $2n$.  Then $\pi_{\alpha_n}$ satisfies Prasad's pole
inequalities, but
\[
  (\pi_{\alpha_n})_{N,\psi}\neq0.
\]
For $\rho=\nu^{-(2n-1)/2}$ these are counterexamples inside the principal
series containing the trivial representation.
\end{corollary}

\begin{proof}
A computation following the definition of \(h_{\alpha_n}\) gives
\[
  h_{\alpha_n}(-1)=h_{\alpha_n}(0)=n,
  \qquad
  \max_t h_{\alpha_n}(t)=n.
\]
Corollary~\ref{cor:trivial-principal-series-tjm} gives
$(\pi_{\alpha_n})_{N,\psi}\neq0$.

The pole counts are
\[
  c_{\alpha_n}(1)=2n-2,
  \qquad
  c_{\alpha_n}(t)=2n-t-2\quad(2\leq t\leq n).
\]
 For $n\geq3$ the
result satisfies $c_{\alpha_n}(t)\geq n-t+1$ for all $1\leq t\leq n$.
\end{proof}



\section{Examples}
\label{sec:further-examples}

We record a few examples which were mentioned as difficult for the
methods of \cite[\S 6.5]{HV26}.  From the present point of view they are
straightforward.

\begin{proposition}[Speh and Arthur-type rectangles]
\label{prop:speh-arthur-examples}
Let \(\pi=L(\mathfrak m)\) be an irreducible representation whose
Zelevinsky dual multisegment is a sum of rectangular blocks
\[
  \mathfrak m^{\mathrm t}
  =
  \sum_i\sum_{j=0}^{\ell_i-1}
  [a_i+j,a_i+j+t_i-1]_{\rho_i}.
\]
Then
\[
  \lambda(\pi)^\top
  =
  \text{the decreasing rearrangement of the integers }
  t_i
  \text{ repeated } \ell_i d_{\rho_i}\text{ times}.
\]
In particular, if \(\pi\) is a representation of \(\GL_{2n}(F)\), then
\[
  \pi_{N,\psi}=0
  \quad\Longleftrightarrow\quad
  \max_i t_i\geq n+1.
\]
\end{proposition}

\begin{proof}
Each segment in the \(i\)-th rectangular block has length \(t_i\), and it
contributes this length with multiplicity \(d_{\rho_i}\) to
\(\lambda(\pi)^\top\) by Theorem~\ref{thm:MW-partition}.  The vanishing
criterion is then Theorem~\ref{thm:corrected}.
\end{proof}

This includes the usual Speh examples.  If
\(\delta=\delta([a,a+\ell-1]_\rho)\) is essentially square-integrable and
\(U(\delta,t)\) denotes the Speh representation built from \(t\) consecutive
twists of \(\delta\), then, up to a harmless common unramified twist,
\[
  \mathfrak m\bigl(U(\delta,t)\bigr)^{\mathrm t}
  =
  \sum_{j=0}^{\ell-1}[a+j,a+j+t-1]_\rho.
\]
Thus
\[
  \lambda\bigl(U(\delta,t)\bigr)^\top
  =
  (\underbrace{t,\dots,t}_{\ell d_\rho\text{ times}}),
  \qquad
  \lambda\bigl(U(\delta,t)\bigr)
  =
  (\underbrace{\ell d_\rho,\dots,\ell d_\rho}_{t\text{ times}}).
\]
If \(\ell d_\rho t=2n\), then
\[
  U(\delta,t)_{N,\psi}=0
  \quad\Longleftrightarrow\quad
  t\geq n+1.
\]
The discrete-series case is \(t=1\), hence gives nonzero twisted Jacquet
module.  The trivial representation of \(\GL_{2n}(F)\) is the opposite
extreme, with \(\ell=d_\rho=1\) and \(t=2n\), hence has zero twisted Jacquet
module.  More generally, let
\[
  \Psi=\bigoplus_i \rho_i\boxtimes S_{\ell_i}\boxtimes S_{t_i}
\]
be an Arthur parameter for some \(\GL_N(F)\) in the sense of Arthur
\cite[\S1.3]{Art13}, where \(S_r\) denotes the
\(r\)-dimensional irreducible representation of \(\operatorname{SL}_2(\mathbb C)\).  We use
the same symbol \(\rho_i\) for the supercuspidal representation corresponding
to the first factor under the local Langlands correspondence.  For \(\GL_N\),
the Arthur packet attached to \(\Psi\) is a singleton; we call its member the
Arthur-type representation attached to \(\Psi\).  With the present
normalization, and up to a harmless common unramified twist, the summand
\(\rho_i\boxtimes S_{\ell_i}\boxtimes S_{t_i}\) contributes the rectangular
piece
\[
  \sum_{j=0}^{\ell_i-1}[a_i+j,a_i+j+t_i-1]_{\rho_i}
\]
to the Zelevinsky dual multisegment.  Then the preceding proposition applies
directly: the relevant integers are the segment lengths \(t_i\).

\begin{proposition}
\label{prop:trivial-parabolic-induction-examples}
Let \(\beta=(\beta_1,\dots,\beta_r)\) be an ordered partition of \(2n\), and
let \(P_\beta=M_\beta U_\beta\) be the standard parabolic subgroup of
\(G=\GL_{2n}(F)\) with Levi size ordered partition \(\beta\).  Let
\[
  \chi=\chi_1\boxtimes\cdots\boxtimes\chi_r
\]
be a one-dimensional representation of
\[
  M_\beta\simeq \GL_{\beta_1}(F)\times\cdots\times\GL_{\beta_r}(F),
\]
inflated to \(P_\beta\), and put
\[
  I_\beta(\chi)=\Ind_{P_\beta}^G(\chi).
\]
Then
\[
  I_\beta(\chi)_{N,\psi}=0
  \quad\Longleftrightarrow\quad
  \max_i\beta_i\geq n+1.
\]
\end{proposition}

\begin{proof}
For each \(i\), since \(\chi_i\) is one-dimensional, there is a character
\(\rho_i\) of \(F^\times\) such that
\[
  \chi_i=Z(\Delta_i),
  \qquad
  \Delta_i=[0,\beta_i-1]_{\rho_i}.
\]
Thus
\[
  I_\beta(\chi)
  =
  Z(\Delta_1)\times\cdots\times Z(\Delta_r).
\]
Put \(\mathfrak m=\Delta_1+\cdots+\Delta_r\).  By Zelevinsky's theorem on
composition factors \cite[Theorem 7.1]{Zel80}, every irreducible composition
factor of \(I_\beta(\chi)\) is of the form \(Z(\mathfrak n)\), where
\(\mathfrak n\) is obtained from \(\mathfrak m\) by a finite sequence of
elementary operations on linked pairs of segments; moreover \(Z(\mathfrak m)\)
itself occurs as a composition factor.  An elementary operation replaces two
linked segments by their union and intersection, so it cannot decrease the
maximum segment length.

If \(\max_i\beta_i\geq n+1\), then every such \(\mathfrak n\) contains a
segment of length at least \(n+1\).  Since
\[
  Z(\mathfrak n)=L(\mathfrak n^{\mathrm t}),
\]
Theorem~\ref{thm:corrected} gives \(Z(\mathfrak n)_{N,\psi}=0\) for every
composition factor.  Exactness of the twisted Jacquet functor then gives
\(I_\beta(\chi)_{N,\psi}=0\).

Conversely, if \(\max_i\beta_i\leq m\), then the composition factor
\(Z(\mathfrak m)\) has nonzero twisted Jacquet module by
Theorem~\ref{thm:corrected}, because all segments of \(\mathfrak m\) have
length at most \(n\).  Exactness again implies
\(I_\beta(\chi)_{N,\psi}\neq0\).
\end{proof}

These examples illustrate the role of our criterion.  Direct
geometric-lemma computations for such induced representations or Speh
representations can be quite involved, but the vanishing of the twisted
Jacquet module is controlled by the largest relevant segment length.


\section{The case of \texorpdfstring{$\GL_4$}{GL4}}
\label{sec:GL4}

\subsection{An explicit counterexample}

Let $\rho$ be a smooth character of $F^\times$ and consider
\[
  \mathfrak m
  =[0,0]_\rho+[1,1]_\rho+[3,3]_\rho+[4,4]_\rho,
  \qquad
  \pi=L(\mathfrak m),
\]
an irreducible representation of $\GL_4(F)$.
For $a\in\mathbb Z$, put
\[
  \eta_a(g)=\rho(\det g)\,|\det g|_F^{a+1/2}
  \qquad(g\in\GL_2(F)).
\]

\begin{proposition}
\label{prop:example}
For $n=2$ and the above representation, as a representation of the
stabilizer $\Delta\GL_2(F)$ of $\psi$ in the Levi subgroup,
\[
  \pi_{N,\psi}\simeq \eta_3\otimes\eta_0.
\]
In particular, $\dim_{\mathbb C}\pi_{N,\psi}=1$.
\end{proposition}

\begin{proof}
The quotient maps
$\nu^{a+1}\rho\times\nu^a\rho\twoheadrightarrow\eta_a$ give
\[
  \nu^4\rho\times\nu^3\rho\times\nu\rho\times\rho
  \twoheadrightarrow
  \eta_3\times\eta_0.
\]
The standard irreducibility criterion
therefore gives
\[
  \pi\simeq\eta_3\times\eta_0.
\]
Applying Corollary~\ref{cor:unique-nonzero-parabolic-subquotient} with
\(n=2\), \(\pi_1=\eta_3\), and \(\pi_2=\eta_0\), we get
\[
  \pi_{N,\psi}\simeq(\eta_3\times\eta_0)_{N,\psi}
  \simeq\eta_3\otimes\eta_0.
\]
The last representation is one-dimensional.
\end{proof}

We now test the pole-to-vanishing implication in Conjecture~\ref{conj:prasad}.

\begin{proposition}
\label{prop:poles}
For the representation in Proposition \ref{prop:example},
\[
  \pord_{s=1}L(s,\pi\times\pi^\vee)=2,
  \qquad
  \pord_{s=2}L(s,\pi\times\pi^\vee)=1.
\]
Consequently, the pole conditions in Prasad's conjecture hold for $\pi$, but
$\pi_{N,\psi}\neq0$.  
\end{proposition}
Thus Proposition this is a counterexample
to the conjecture.

\begin{proof}
The Langlands parameter of $\pi$ is
\[
  \phi_\pi
  =\rho\nu^0\oplus\rho\nu^1\oplus\rho\nu^3\oplus\rho\nu^4.
\]
With $S=\{0,1,3,4\}$, multiplicativity gives
\[
  L(s,\pi\times\pi^\vee)
  =\prod_{a,b\in S}L(s+a-b,\one).
\]
The factor $L(s+a-b,\one)$ has a real pole at $s=t$ exactly when $b-a=t$.
 The asserted pole orders follow.  Proposition
\ref{prop:example} gives the contradiction to the proposed vanishing
criterion.
\end{proof}

\subsection{All constituents of the principal series containing the trivial representation}

We spell out the case $m=4$.  Write $[a]_\rho=[a,a]_\rho$ and
\[
  \pi_\alpha=L(\mathfrak p_\alpha).
\]
Put
\[
  \kappa=\rho^2\nu^3.
\]
The eight subquotients are as follows:
\[
\begin{array}{c|c|c|c|c}
\alpha & \mathfrak p_\alpha & \mathfrak p_\alpha^{\mathrm t}
& r_{\MW} & (\pi_\alpha)_{N,\psi}\\ \hline
(1,1,1,1) &
[0]_\rho+[1]_\rho+[2]_\rho+[3]_\rho &
[0,3]_\rho & 4 & 0\\
(2,1,1) &
[0,1]_\rho+[2]_\rho+[3]_\rho &
[1,3]_\rho+[0]_\rho & 3 & 0\\
(1,2,1) &
[0]_\rho+[1,2]_\rho+[3]_\rho &
[0,1]_\rho+[2,3]_\rho & 2 & \neq0\\
(1,1,2) &
[0]_\rho+[1]_\rho+[2,3]_\rho &
[0,2]_\rho+[3]_\rho & 3 & 0\\
(2,2) &
[0,1]_\rho+[2,3]_\rho &
[1,2]_\rho+[0]_\rho+[3]_\rho & 2 & \neq0\\
(3,1) &
[0,2]_\rho+[3]_\rho &
[2,3]_\rho+[1]_\rho+[0]_\rho & 2 & \neq0\\
(1,3) &
[0]_\rho+[1,3]_\rho &
[0,1]_\rho+[2]_\rho+[3]_\rho & 2 & \neq0\\
(4) &
[0,3]_\rho &
[0]_\rho+[1]_\rho+[2]_\rho+[3]_\rho & 1 & \neq0.
\end{array}
\]

We now compute the nonzero modules.  Put
\[
  \delta_i=\delta([i,i+1]_\rho),
  \qquad
  \eta_i=Z([i,i+1]_\rho)
  \quad(i=0,2),
\]
so that $\eta_i(g)=\rho(\det g)|\det g|_F^{i+1/2}$ is a character of
$\GL_2(F)$.  Tensor products in the following table are tensor products as
representations of the diagonal $\GL_2(F)$:
\[
\begin{array}{c|c|c}
\alpha & \pi_\alpha & (\pi_\alpha)_{N,\psi}\\ \hline
(1,2,1) & L([0]_\rho+[1,2]_\rho+[3]_\rho)
& \eta_2\otimes\eta_0\\
(2,2) & L([0,1]_\rho+[2,3]_\rho)
& (\kappa\nu^{3/2})\times(\kappa\nu^{-3/2})\\
(3,1) & L([0,2]_\rho+[3]_\rho)
& \eta_2\otimes\delta_0\\
(1,3) & L([0]_\rho+[1,3]_\rho)
& \delta_2\otimes\eta_0\\
(4) & L([0,3]_\rho)
& \delta_2\otimes\delta_0.
\end{array}
\]
Here $(\kappa\nu^{3/2})\times(\kappa\nu^{-3/2})$ denotes normalized
induction from the Borel subgroup of $\GL_2(F)$.  In the untwisted case
$\rho=\nu^{-3/2}$, one has $\kappa=\one$, so this term is
$\nu^{3/2}\times\nu^{-3/2}$.

Indeed, the short exact sequences
\[
\begin{gathered}
0\to L([0]_\rho+[1,2]_\rho+[3]_\rho)
  \to \eta_2\times\eta_0
  \to L([0]_\rho+[1]_\rho+[2]_\rho+[3]_\rho)\to0,\\
0\to L([0]_\rho+[1,3]_\rho)
  \to \delta_2\times\eta_0
  \to L([0]_\rho+[1]_\rho+[2,3]_\rho)\to0,\\
0\to L([0,2]_\rho+[3]_\rho)
  \to \eta_2\times\delta_0
  \to L([0,1]_\rho+[2]_\rho+[3]_\rho)\to0,\\
0\to L([0,3]_\rho)
  \to \delta_2\times\delta_0
  \to L([0,1]_\rho+[2,3]_\rho)\to0
\end{gathered}
\]
come from Zelevinsky's length-two reducibility criterion.  Since \(N\) is
unipotent, the twisted Jacquet functor \((-)_{N,\psi}\) is exact.  The
\(\GL_4\) calculation for \(P_{2,2}\)-induced representations
\cite[Proposition 5.21]{HV26} gives
\[
\begin{aligned}
  (\eta_2\times\eta_0)_{N,\psi}&\simeq\eta_2\otimes\eta_0,\\
  (\delta_2\times\eta_0)_{N,\psi}&\simeq\delta_2\otimes\eta_0,\\
  (\eta_2\times\delta_0)_{N,\psi}&\simeq\eta_2\otimes\delta_0,
\end{aligned}
\]
and an exact sequence
\[
  0\to\delta_2\otimes\delta_0
  \to(\delta_2\times\delta_0)_{N,\psi}
  \to(\kappa\nu^{3/2})\times(\kappa\nu^{-3/2})\to0.
\]
 In
\cite[Proposition 5.21(2)]{HV26}, the last term is the normalized induction
to \(\GL_2(F)\) from its Borel subgroup of the two characters
\[
  \rho^2\nu^5\cdot\nu^{-1/2}=\kappa\nu^{3/2},
  \qquad
  \rho^2\nu\cdot\nu^{1/2}=\kappa\nu^{-3/2}.
\]
Together with the vanishing entries in the first table, these formulas give
the five displayed twisted Jacquet modules.
For $\rho=\nu^{-3/2}$, this is the principal series $\xi$ considered in
\cite[\S5.5]{HV26}.  The rows above correspond to the following notation in
that paper:
\[
\begin{array}{c|c}
\alpha & \text{notation in \cite[\S5.5]{HV26}}\\ \hline
(1,1,1,1) & \one_4\\
(2,1,1) & Q_{\nu,2}^{\vee}\\
(1,2,1) & L_{\nu^{-1},2}\\
(1,1,2) & Q_{\nu,2}\\
(2,2) & \tau\\
(3,1) & Z_{\nu,2}^{\vee}\\
(1,3) & Z_{\nu,2}\\
(4) & \operatorname{St}_4.
\end{array}
\]
Thus the vanishing and nonvanishing entries agree with the computations in
\cite[\S5.5]{HV26}.


\section{The case of \texorpdfstring{$\GL_6$}{GL6}}
\label{sec:GL6}

\subsection{A nonzero twisted Jacquet module}

\begin{proposition}
\label{prop:second-example}
Let $\rho$ be a smooth character of $F^\times$ and let
\[
  \mathfrak m_6
  =[0,0]_\rho+[0,0]_\rho+[1,1]_\rho+[2,3]_\rho+[3,3]_\rho,
  \qquad
  \pi_6=L(\mathfrak m_6),
\]
an irreducible representation of $\GL_6(F)$.
Then, as a representation of the stabilizer $\Delta\GL_3(F)$ of $\psi$,
\[
  (\pi_6)_{N,\psi}
  \simeq
  \rho^2\nu\times\rho^2\nu^4\times\rho^2\nu^4.
\]
\end{proposition}

\begin{proof}
The M\oe glin--Waldspurger algorithm gives
\[
  \mathfrak m_6^{\mathrm t}
  =[0,2]_\rho+[0,0]_\rho+[3,3]_\rho+[3,3]_\rho.
\]
 Hence
\[
  r_{\MW}(\mathfrak m_6)=3.
\]
Since $n=3$, Theorem \ref{thm:corrected} gives
$(\pi_6)_{N,\psi}\neq0$.

We now determine the $\Delta\GL_3(F)$-module.  Since
$\pi_6=L(\mathfrak m_6)=Z(\mathfrak m_6^{\mathrm t})$, it is the unique
irreducible submodule of
\[
  \Theta
  =
  \nu^3\rho\times\nu^3\rho\times\rho\times Z([0,2]_\rho).
\]
Here $Z([0,2]_\rho)$ is a character of $\GL_3(F)$.  Applying
\cite[Proposition 5.2]{HV26} to this \(P_{3,3}\)-induced representation gives
\[
  \Theta_{N,\psi}
  \simeq
  \bigl(\nu^3\rho\times\nu^3\rho\times\rho\bigr)\otimes Z([0,2]_\rho)
  \simeq
  Z([0,2]_\rho)\otimes
  \bigl(\rho\times\nu^3\rho\times\nu^3\rho\bigr).
\]

It remains to see that the quotient $\Theta/\pi_6$ contributes nothing.  In
Zelevinsky's description of composition factors by elementary operations
\cite[Theorem 7.1]{Zel80}, every composition factor of \(\Theta\) other than
\(\pi_6\) is obtained from \(\mathfrak m_6^{\mathrm t}\) by applying at least
one elementary operation to one of the linked pairs
\([0,2]_\rho\), \([3,3]_\rho\).  It follows that the Zelevinsky
multisegment of every irreducible constituent of $\Theta/\pi_6$ contains a
segment of length $4$.  For such a factor $\tau=Z(\mathfrak n)$, equivalently
$\tau=L(\mathfrak n^{\mathrm t})$, Theorem~\ref{thm:corrected} gives
$\tau_{N,\psi}=0$, since here $n=3$.  Exactness of twisted Jacquet functors therefore identifies
$(\pi_6)_{N,\psi}$ with $\Theta_{N,\psi}$, proving the stated formula.
\end{proof}

\subsection{Parabolic indexing and Richardson orbits}

For $m=6$, the following table records the ordered partition \(\alpha\)
indexing \(\pi_\alpha=L(\mathfrak p_\alpha)\), the ordered partition
\(\beta=\alpha^\complement\) such that
\(v_{P_\beta}^G\simeq \pi_{\alpha}\), the partition
\(\lambda_\alpha=(\beta^\downarrow)^\top\) attached to the Richardson orbit,
and the corresponding twisted Jacquet module.

\begin{center}
\small
\begin{tabular}{c|c|c|c}
$\alpha$ & $\beta=\alpha^\complement$
& $\lambda_\alpha$ & $(\pi_\alpha)_{N,\psi}$\\ \hline
$(1,1,1,1,1,1)$ & $(6)$ & $(1,1,1,1,1,1)$ & \(0\)\\
$(1,1,1,1,2)$ & $(5,1)$ & $(2,1,1,1,1)$ & \(0\)\\
$(1,1,1,2,1)$ & $(4,2)$ & $(2,2,1,1)$ & \(0\)\\
$(1,1,1,3)$ & $(4,1,1)$ & $(3,1,1,1)$ & \(0\)\\
$(1,1,2,1,1)$ & $(3,3)$ & $(2,2,2)$ & \(\neq0\)\\
$(1,1,2,2)$ & $(3,2,1)$ & $(3,2,1)$ & \(\neq0\)\\
$(1,1,3,1)$ & $(3,1,2)$ & $(3,2,1)$ & \(\neq0\)\\
$(1,1,4)$ & $(3,1,1,1)$ & $(4,1,1)$ & \(\neq0\)\\
$(1,2,1,1,1)$ & $(2,4)$ & $(2,2,1,1)$ & \(0\)\\
$(1,2,1,2)$ & $(2,3,1)$ & $(3,2,1)$ & \(\neq0\)\\
$(1,2,2,1)$ & $(2,2,2)$ & $(3,3)$ & \(\neq0\)\\
$(1,2,3)$ & $(2,2,1,1)$ & $(4,2)$ & \(\neq0\)\\
$(1,3,1,1)$ & $(2,1,3)$ & $(3,2,1)$ & \(\neq0\)\\
$(1,3,2)$ & $(2,1,2,1)$ & $(4,2)$ & \(\neq0\)\\
$(1,4,1)$ & $(2,1,1,2)$ & $(4,2)$ & \(\neq0\)\\
$(1,5)$ & $(2,1,1,1,1)$ & $(5,1)$ & \(\neq0\)\\
$(2,1,1,1,1)$ & $(1,5)$ & $(2,1,1,1,1)$ & \(0\)\\
$(2,1,1,2)$ & $(1,4,1)$ & $(3,1,1,1)$ & \(0\)\\
$(2,1,2,1)$ & $(1,3,2)$ & $(3,2,1)$ & \(\neq0\)\\
$(2,1,3)$ & $(1,3,1,1)$ & $(4,1,1)$ & \(\neq0\)\\
$(2,2,1,1)$ & $(1,2,3)$ & $(3,2,1)$ & \(\neq0\)\\
$(2,2,2)$ & $(1,2,2,1)$ & $(4,2)$ & \(\neq0\)\\
$(2,3,1)$ & $(1,2,1,2)$ & $(4,2)$ & \(\neq0\)\\
$(2,4)$ & $(1,2,1,1,1)$ & $(5,1)$ & \(\neq0\)\\
$(3,1,1,1)$ & $(1,1,4)$ & $(3,1,1,1)$ & \(0\)\\
$(3,1,2)$ & $(1,1,3,1)$ & $(4,1,1)$ & \(\neq0\)\\
$(3,2,1)$ & $(1,1,2,2)$ & $(4,2)$ & \(\neq0\)\\
$(3,3)$ & $(1,1,2,1,1)$ & $(5,1)$ & \(\neq0\)\\
$(4,1,1)$ & $(1,1,1,3)$ & $(4,1,1)$ & \(\neq0\)\\
$(4,2)$ & $(1,1,1,2,1)$ & $(5,1)$ & \(\neq0\)\\
$(5,1)$ & $(1,1,1,1,2)$ & $(5,1)$ & \(\neq0\)\\
$(6)$ & $(1,1,1,1,1,1)$ & $(6)$ & \(\neq0\)\\
\end{tabular}
\end{center}

The last column follows from Corollary~\ref{cor:trivial-principal-series-tjm}
and Proposition~\ref{prop:diagonal-complement}: since \(n=3\), the twisted
Jacquet module vanishes exactly when \(\max_i\beta_i\geq4\).  
Here the wave-front set of $v_{P_\beta}^G$ is the analytic closure of the
Richardson orbit.  Parabolics
with the same unordered Levi sizes have the same partition attached to the
wave-front set, even when the corresponding representations are different.

\subsection{The three failures of the pole equivalence}

\begin{proposition}[The three failures in \(\GL_6\)]
\label{prop:GL6-three-failures}
Among the \(32\) constituents \(\pi_\alpha\) of \(\Sigma_6(\rho)\), exactly
three violate the proposed equivalence in Conjecture~\ref{conj:prasad}.  They
are
\[
\begin{array}{c|c|c|c}
\alpha & (c_\alpha(1),c_\alpha(2),c_\alpha(3))
& \max_t h_\alpha(t) & (\pi_\alpha)_{N,\psi}\\ \hline
(1,1,2,1,1) & (4,2,1) & 3 & \neq0\\
(1,1,2,2) & (3,2,1) & 3 & \neq0\\
(2,2,1,1) & (3,2,1) & 3 & \neq0.
\end{array}
\]
Thus all three failures are failures of the pole-to-vanishing implication.
\end{proposition}

\begin{proof}
Using Proposition~\ref{prop:principal-series-poles} and Corollary
\ref{cor:trivial-principal-series-tjm}, one checks the \(32\) ordered partitions of
\(6\) by computing \((c_\alpha(1),c_\alpha(2),c_\alpha(3))\) and
\(\max_t h_\alpha(t)\).  The displayed rows are exactly the rows for which the
pole inequalities hold while the exact criterion gives nonvanishing.  Indeed,
in each displayed row
$(c_\alpha(1),c_\alpha(2),c_\alpha(3))\geq(3,2,1)$, whereas
\(\max_t h_\alpha(t)=3<n+1=4\).
Hence $29$ of the $32$ constituents satisfy the full equivalence, and each of
the three exceptions has the pole inequalities but a nonzero twisted Jacquet
module.  There are no violations in the opposite
direction, in agreement with Theorem~\ref{thm:general-vanishing-to-poles}.
\end{proof}

The rarity of the failures is worth emphasizing.  The pole profile detects
vanishing correctly for most constituents in this small rank, but it is
strictly coarser than the Zelevinsky dual invariant.  The three rows
above show exactly where the extra information in the wave-front set is lost.



\section{Vanishing criterion from \texorpdfstring{$L$}{L}-functions}
\label{sec:exterior-square}

This section considers two different uses of \(L\)-functions.  Exterior-square
factors detect possibly twisted Shalika functionals and therefore give a
sufficient condition for \(\pi_{N,\psi}\neq0\), while the segment-length ratio
below gives an exact vanishing criterion, which is a reformulation of Theorem \ref{thm:corrected}.

\subsection{Exterior-square \texorpdfstring{$L$}{L}-factors}

The standard exterior-square criterion is a criterion for Shalika functionals.
Let
\[
  S_{2n}=\Delta\GL_n(F)\ltimes N.
\]
For a character \(\eta\) of \(F^\times\), extend \(\psi\) to \(S_{2n}\) by
\[
  \psi_\eta(\Delta(g)n(X))=\eta(\det g)\psi(n(X)).
\]
Then
\[
  \Hom_{S_{2n}}(\pi,\psi_\eta)
  \simeq
  \Hom_{\Delta\GL_n(F)}
  \bigl(\pi_{N,\psi},\eta\circ\det\bigr).
\]
Thus a twisted Shalika model is a character quotient of the
\(\Delta\GL_n(F)\)-module \(\pi_{N,\psi}\).  It implies
\(\pi_{N,\psi}\neq0\), but it is stronger than nonvanishing.

This is the representation-theoretic meaning of the exterior-square factor:
the condition that \(\wedge^2\phi_\pi\) contain \(\eta\) is detected by a pole
of \(L(s,\pi,\wedge^2\otimes\eta^{-1})\).  For square-integrable
representations of \(\GL_{2m}(F)\), this is the Jacquet--Shalika criterion for
the Shalika model; see \cite[Proposition 6.1]{Mat12}.  For generalized
Steinberg representations, the corresponding criterion is recorded in
\cite[Theorem 6.1]{Mat12}.  We use the agreement of the Jacquet--Shalika and
Langlands-parameter exterior-square factors as in Kewat--Raghunathan
\cite{KR12}.

For the principal-series constituents containing the trivial representation,
the examples above show that this cannot be promoted to a criterion for
\(\pi_{N,\psi}\).  In \(\GL_4\), the two constituents indexed by
\(\alpha=(2,1,1)\) and \(\alpha=(2,2)\) have the same full exterior-square
pole divisor, even after allowing unramified character twists:
\[
  \pord_{s=t}L(s,\pi_\alpha,\wedge^2)=1
  \quad(t=-2,-1,0,2),
\]
and no other poles.  But
\[
  (\pi_{(2,1,1)})_{N,\psi}=0,
  \qquad
  (\pi_{(2,2)})_{N,\psi}\neq0.
\]

\subsection{The segment-length \texorpdfstring{\(L\)}{L}-ratio criterion}

An exact \(L\)-function criterion must detect the lengths of the segments in the
Zelevinsky dual multisegment.  For a cuspidal line represented by \(\rho\) and
an integer \(r\geq1\), set
\[
  \tau_{\rho,r}:=\delta([0,r-1]_\rho).
\]
If \(\pi^{\mathrm t}\) is the Zelevinsky dual of \(\pi\), define, for
\(r\geq2\),
\begin{equation}
\label{eq:segment-length-ratio}
  \mathcal R_{\rho,r}(s,\pi)
  :=
  \frac{L(s,\pi^{\mathrm t}\times\tau_{\rho,r}^{\vee})}
       {L(s,\pi^{\mathrm t}\times\tau_{\rho,r-1}^{\vee})}.
\end{equation}

\begin{proposition}
\label{prop:segment-length-ratio-factorization}
Let \(\pi=L(\mathfrak m)\) be irreducible.  Write the part of
\(\mathfrak m^{\mathrm t}\) on the cuspidal line represented by \(\rho\) as
\[
  (\mathfrak m^{\mathrm t})_{[\rho]}=\sum_i [a_i,b_i]_\rho,
  \qquad \ell_i=b_i-a_i+1.
\]
Then
\begin{equation}
\label{eq:segment-length-ratio-product}
  \mathcal R_{\rho,r}(s,\pi)
  =
  \prod_{i:\,\ell_i\geq r}
  L(s+b_i-r+1,\rho\times\rho^\vee).
\end{equation}
In particular, \(\mathcal R_{\rho,r}(s,\pi)\equiv1\) if and only if every
segment of \(\mathfrak m^{\mathrm t}\) on the cuspidal line \([\rho]\) has
length \(<r\).
\end{proposition}

\begin{proof}
Apply Proposition~\ref{prop:RS} to a segment \([a_i,b_i]_\rho\) on the
cuspidal line \([\rho]\) and to \([0,r-1]_\rho\).  Passing from \(r-1\) to
\(r\) adds exactly one factor, namely
\(L(s+b_i-r+1,\rho\times\rho^\vee)\), and this happens precisely when
\(\ell_i\geq r\).  Multiplying over all segments on \([\rho]\) gives
\eqref{eq:segment-length-ratio-product}; the remaining cuspidal lines do not
contribute to the Rankin--Selberg factor with \(\tau_{\rho,r}\).
\end{proof}

\begin{theorem}
\label{thm:segment-length-ratio-criterion}
Let \(\pi\) be an irreducible admissible representation of \(\GL_{2n}(F)\).
Then
\[
  \pi_{N,\psi}=0
  \quad\Longleftrightarrow\quad
  \mathcal R_{\rho,n+1}(s,\pi)\not\equiv1
  \text{ for some cuspidal line }[\rho].
\]
Equivalently, the quotient
\[
  \frac{L(s,\pi^{\mathrm t}\times\tau_{\rho,n+1}^{\vee})}
       {L(s,\pi^{\mathrm t}\times\tau_{\rho,n}^{\vee})}
\]
has a pole for some \(\rho\).
\end{theorem}

\begin{proof}
By Proposition~\ref{prop:segment-length-ratio-factorization}, for a fixed
cuspidal line \([\rho]\) the quotient is nontrivial exactly when
\(\mathfrak m^{\mathrm t}\) contains a segment of length at least \(n+1\) on
\([\rho]\).  Taking all cuspidal lines together, this is
Theorem~\ref{thm:corrected}.
\end{proof}

Now fix \(m\), a smooth character \(\rho\) of \(F^\times\), and an ordered partition
\(\alpha\) of \(m\).  Let \(\pi_\alpha=L(\mathfrak p_\alpha)\) be the
corresponding constituent of the principal series \(\Sigma_m(\rho)\) from
\eqref{eq:sigma-m-rho}--\eqref{eq:pi-alpha}.  For these representations the
criterion becomes particularly transparent.  Let
\(\beta=\alpha^\complement\).  By Proposition~\ref{prop:diagonal-complement},
the segment lengths in \(\mathfrak p_\alpha^{\mathrm t}\) are exactly the parts
of \(\beta\).  Therefore Proposition~\ref{prop:segment-length-ratio-factorization}
gives, for \(r\geq2\),
\[
  \mathcal R_{\rho,r}(s,\pi_\alpha)\not\equiv1
  \quad\Longleftrightarrow\quad
  \max_i\beta_i\geq r.
\]
 Hence, when \(m=2n\),
\[
  (\pi_\alpha)_{N,\psi}=0
  \quad\Longleftrightarrow\quad
  \mathcal R_{\rho,n+1}(s,\pi_\alpha)\not\equiv1
  \quad\Longleftrightarrow\quad
  \max_i\beta_i\geq n+1.
\]
This is the same condition as \(\max_t h_\alpha(t)\geq n+1\), by
Proposition~\ref{prop:diagonal-complement}.

For the constituent \(\pi_\alpha=L(\mathfrak p_\alpha)\) of
\(\Sigma_m(\rho)\) in \eqref{eq:sigma-m-rho}, one can also formulate a criterion using
\(\pi_\alpha\) itself rather than its Zelevinsky dual.  Let
\[
  0=s_0<s_1<\cdots<s_k=m
\]
be obtained from \(B(\alpha)\) by adjoining \(0\) and \(m\).
Define the reduced standard factor
\[
  L_m^\circ(s,\pi_\alpha;\rho)
  :=
  \frac{L(s,\pi_\alpha\times\rho^\vee)}
       {L(s+m-1,\one)}.
\]
For \(r\geq2\), put
\[
  \mathcal P_{\rho,r}(s,\pi_\alpha)
  :=
  \prod_{i=0}^{r-2}L_m^\circ(s+i,\pi_\alpha;\rho).
\]
When \(m=2n\), we also write
\[
  L^\circ(s,\pi_\alpha;\rho):=L_{2n}^\circ(s,\pi_\alpha;\rho)
  =
  \frac{L(s,\pi_\alpha\times\rho^\vee)}
       {L(s+2n-1,\one)}.
\]

\begin{proposition}
\label{prop:reduced-standard-factor-criterion}
Let \(\beta=\alpha^\complement\).  For \(r\geq2\), the following are
equivalent:
\begin{itemize}
\item[(1)]\[
  \max_i\beta_i\geq r;
\]
\item[(2)]\[
  L_m^\circ(s,\pi_\alpha;\rho)
  \text{ contains a string of }r-1\text{ consecutive integral poles;}
\]
\item[(3)]
\[
  \mathcal P_{\rho,r}(s,\pi_\alpha)
  \text{ has a pole of order }r-1\text{ at some integer }s.
\]
\end{itemize}
In particular, when \(m=2n\),
\[
  (\pi_\alpha)_{N,\psi}=0
  \quad\Longleftrightarrow\quad
  \prod_{i=0}^{n-1}L^\circ(s+i,\pi_\alpha;\rho)
  \text{ has a pole of order }n
  \text{ at some integer }s.
\]
\end{proposition}

\begin{proof}
By multiplicativity of local \(L\)-factors and Proposition~\ref{prop:RS},
applied to each segment
\([s_{j-1},s_j-1]_\rho\) and to \([0,0]_\rho\),
\[
  L(s,\pi_\alpha\times\rho^\vee)
  =
  \prod_{j=1}^k L(s+s_j-1,\one).
\]
The denominator in \(L_m^\circ\) removes the factor coming from \(s_k=m\).
Hence the pole set of \(L_m^\circ(s,\pi_\alpha;\rho)\) is
\[
  \{1-b:b\in B(\alpha)\}.
\]
All these poles are simple.

By definition, \(B(\beta)=\{1,\dots,m-1\}\smallsetminus B(\alpha)\).  A part
of \(\beta\) has size at least \(r\) exactly when two consecutive break points
of \(\beta\), allowing also \(0\) and \(m\), are separated by at least \(r\).
This is equivalent to saying that \(B(\alpha)\) contains \(r-1\) consecutive
integers.  After applying \(b\mapsto1-b\), this is the same as saying that
\(L_m^\circ(s,\pi_\alpha;\rho)\) contains a string of \(r-1\) consecutive
integral poles.

Finally, since the poles of \(L_m^\circ\) are simple,
\(\mathcal P_{\rho,r}(s,\pi_\alpha)\) has a pole of order \(r-1\) at an
integer \(u\) exactly when
\[
  L_m^\circ(s,\pi_\alpha;\rho)
  \text{ has poles at }u,u+1,\dots,u+r-2.
\]
This proves the three equivalent conditions.  For \(m=2n\), the last assertion
follows by taking \(r=n+1\) and using Corollary
\ref{cor:trivial-principal-series-tjm}, or equivalently
\(\max_i\beta_i\geq n+1\); here
\(\mathcal P_{\rho,n+1}(s,\pi_\alpha)
=\prod_{i=0}^{n-1}L^\circ(s+i,\pi_\alpha;\rho)\).
\end{proof}

\begin{remark}
The consecutive-support hypothesis in Proposition
\ref{prop:reduced-standard-factor-criterion} is essential.  The same reduced
standard factor does not give a criterion for an arbitrary multisegment.
\end{remark}

\bibliographystyle{alpha}
\bibliography{biblio}

\end{document}